\documentclass[11pt]{amsart}
\usepackage{amssymb, hyperref}
\usepackage{mathtools}
\usepackage{mathabx}
\usepackage{enumerate}

\numberwithin{equation}{section}
\DeclarePairedDelimiter\abs{\lvert}{\rvert}%
\DeclarePairedDelimiter\norm{\lVert}{\rVert}%
\DeclarePairedDelimiter\bnorm{\biggl\lVert}{\biggr\rVert}%
\DeclarePairedDelimiter\babs{\bigl\lvert}{\bigr\rvert}%
\makeatletter
\let\oldabs\abs
\def\abs{\@ifstar{\oldabs}{\oldabs*}}
\let\oldnorm\norm
\def\norm{\@ifstar{\oldnorm}{\oldnorm*}}
\makeatother

\theoremstyle{plain}
\newtheorem{thm}{Theorem}[section]
\newtheorem{theorem}{Theorem}
\newtheorem*{assume}{Assumption}

\newtheorem{lem}[thm]{Lemma}
\newtheorem{prop}[thm]{Proposition}
\newtheorem{cor}[thm]{Corollary}
\theoremstyle{definition}

\theoremstyle{remark}

\newcommand\R{\mathbb{R}}
\newcommand\Z{\mathbb{Z}}
\newcommand\N{\mathbb{N}}

\newcommand\D{\mathcal{D}}

\newcommand\E{\mathcal{E}}

\newcommand\W{\mathcal{W}}

\newcommand\ps{\mathfrak{S}}
\newcommand\qnk{{\mathcal{Q}^n_{t} }}

\newcommand{\ben}{\begin{enumerate}[(i)]}
\newcommand{\een}{\end{enumerate}}
\newcommand{\ft}{\mathcal{F}}

\newcommand{\ift}{\mathcal{F}^{-1}}

\newcommand{\wh}{\widehat}
\newcommand{\les}{\lesssim}
\newcommand{\ges}{\gtrsim}
\newcommand{\supp}{\operatorname{supp}}
\newcommand{\meas}{\operatorname{meas}}

\newcommand{\inn}[1]{\langle #1 \rangle}
\newcommand{\sumab}[2]{\sum_{\substack{ {#1} \\ {#2} }}}

\begin{document}
\title[Spectral multipliers on compact manifolds]{Endpoint bounds for a class of spectral multipliers on compact manifolds}
\author{Jongchon Kim}
\address{Department of Mathematics, University of Wisconsin-Madison, Madison, WI 53706 USA}
\email{jkim@math.wisc.edu}
\begin{abstract}
It is well known that the Stein-Tomas $L^2$ Fourier restriction theorem can be used to derive sharp $L^p$ bounds for radial Fourier multipliers such as the Bochner-Riesz means. In a similar manner, $L^p \to L^2$ estimates for spectral projection operators have been utilized in order to obtain sharp $L^p$ bounds for spectral multipliers of self-adjoint elliptic pseudo-differential operators on compact manifolds. In this paper, we refine an endpoint result for spectral multipliers due to Seeger, providing endpoint bounds in terms of Besov spaces. Our proof is based on the ideas from the recent work by Heo, Nazarov and Seeger, and Lee, Rogers and Seeger on radial Fourier multipliers.
\end{abstract}
\maketitle
\section{Introduction}
Assume that $M$ is a compact smooth manifold of dimension $d\geq 2$ without boundary. Let $A$ be a first order classical elliptic pseudo-differential operator on $M$ which is positive and self-adjoint with respect to a smooth positive density $dx$. An important special case is $\sqrt{-\Delta}$ for the Laplace-Beltrami operator $\Delta$ on a Riemannian manifold. For the background information on pseudo-differential operators on manifolds and related topics, we refer the reader to \cite{Shubin, Taylor, Sogge}.

It can be shown by spectral theory that $L^2(M)$ admits a spectral decomposition and the spectrum of $A$ is discrete; $0\leq \lambda_1 \leq \lambda_2 \cdots$. Let $E_l$ be the orthogonal projection onto the eigenspace associated with the eigenvalue $\lambda_l$. For each bounded function $m$, one can define a bounded operator $m(A)$ on $L^2(M)$ by
\[m(A)f = \sum_l m(\lambda_l) E_lf. \]

Let $a(x,\xi)$ be the principal symbol of $A$. Throughout the paper, we make the following assumption.
\begin{assume} 
For each $x\in M$, $\Sigma_x = \{ \xi \in T^*_xM : a(x,\xi) = 1 \}$ has everywhere nonvanishing Gaussian curvature.
\end{assume}
As a consequence of the assumption, $L^p \to L^2$ estimates for spectral projection operators $\chi_n:=\chi_{[n,n+1]}(A)$ were obtained for a given $n\geq 0$. Let $\delta(p)=d(\frac{1}{p}-\frac{1}{2})-\frac{1}{2}$, which coincides with the critical index for Bochner-Riesz multipliers in the range $p\leq\frac{2d}{d+1}$.
\begin{theorem} \label{thm:spec}
Let $1 \leq p \leq \frac{2(d+1)}{d+3}$. Then
\begin{equation} \label{eqn:res}
\norm{\chi_n f}_{L^2(M)} \les (1+n)^{\delta(p)} \norm{f}_{L^p(M)}.
\end{equation}
\end{theorem}
This result is due to Sogge \cite{SoggeS}, Christ and Sogge \cite{CS}, and Seeger and Sogge \cite{SS2}. Theorem A can be seen as a generalization of the Stein-Tomas $L^2$ Fourier restriction theorem, which has been successfully applied, among other things, in the study of radial Fourier multipliers such as the Bochner-Riesz multipliers (see e.g. \cite{Fefferman, Christ, SeQ}). In a similar manner, Theorem A has been utilized in order to obtain certain sharp $L^p$ estimates for the operator $m(A)$ (see e.g. \cite{SoggeR, SS}). For more general results in an abstract setting, we refer the reader to \cite{GHS} and references therein.

Endpoint estimates, which often demand more delicate arguments, have also been obtained for the Bochner-Riesz multipliers \cite{ChristW1,ChristW2, SeegerW, TaoW,TaoW2}. In the case of the Riesz means for the eigenfunction expansions for pseudo-differential operators on compact manifolds, i.e. $m^\delta_t(A)$ where $m^\delta_t(\lambda) =  (1-\lambda/t)^\delta_+$, there is the following endpoint result at the critical index $\delta=\delta(p)$. We denote by $L^{p,q}$ the Lorentz spaces.
\begin{theorem} \label{thm:BR}
Let $1 \leq p \leq \frac{2(d+1)}{d+3}$. Then 
\[ \sup_{t>0}\norm{m^{\delta(p)}_t (A)f}_{L^{p,\infty}(M)} \les \norm{f}_{L^p(M)}. \]
\end{theorem}
This result is due to Christ and Sogge \cite{CS}, Seeger \cite{Seeger}, and Tao \cite{TaoW}.
More generally, Seeger \cite{Seeger} proved an endpoint result on $m(A)$ for $m$ in localized $R^2_{\alpha,q}$ spaces, where $R^2_{\alpha,q}$ is a function space which enjoys properties similar to, but is strictly contained in the Besov space $B^2_{\alpha,q}$ for $q>1$. We recall that the Bochner-Riesz multipliers can be decomposed as a sum of multipliers supported in finitely many overlapping thin annuli (see e.g. \cite{Stein}). The space $R^2_{\alpha,q}$ is distinguished from $B^2_{\alpha,q}$ by the existence of such a decomposition, which makes it convenient to exploit orthogonality. It had remained as an open question whether one can replace $R^2_{\alpha,q}$ by $B^2_{\alpha,q}$. 

Lee, Rogers and Seeger \cite{LRS} answered the question in the affirmative at least in the setting of radial Fourier multipliers. They adapted an approach used by Heo, Nazarov and Seeger \cite{HNS,HNS1}, where a necessary and sufficient condition for the $L^p$ boundedness of radial Fourier multipliers was provided in sufficiently high dimensions. The role of the $L^2$ Fourier restriction theorem was crucial in \cite{LRS}. Some of results in \cite{HNS, HNS1,LRS} were generalized by the author \cite{KimMax, Kim} to quasiradial Fourier multipliers and maximal operators associated with them.

The recent developments on radial Fourier multipliers suggest a similar improvement on the estimate for spectral multipliers $m(A)$. The main result of this paper is an endpoint estimate  for $m(A)$ in terms of Besov spaces. This answers the question arising from \cite{Seeger} in the affirmative. In what follows, we let $\alpha(p) = \delta(p)+\frac{1}{2}= d\left(\frac{1}{p}-\frac{1}{2}\right)$. 

\begin{thm}\label{thm:main2}
Let $1 < p < \frac{2(d+1)}{d+3}$, $p\leq q\leq \infty$, and $\psi$ be a non-trivial smooth function compactly supported in $(0,\infty)$. Then
\begin{equation*}
\norm{m(A) f}_{L^{p,q}(M)} \les \sup_{t>0} \norm{m(t\cdot)\psi}_{B^2_{\alpha(p),q}(\R) }\norm{ f}_{L^p(M)}.
\end{equation*} 
\end{thm}
Let us mention some earlier results. Theorem \ref{thm:main2} with $B^2_{\alpha(p),q}$ replaced by the Sobolev space $L^2_{\alpha}$ for any $\alpha>\alpha(p)$ is due to Seeger and Sogge \cite{SS}. Seeger \cite{Seeger} improved their result by replacing the Sobolev space by $R^2_{\alpha(p),q}$ discussed earlier. 

By a transplantation theorem of Mitjagin \cite{Mit} (see also \cite{KST,CS}), we may deduce an endpoint Fourier multiplier theorem from Theorem \ref{thm:main2}. Let $a$ be a smooth positive function on $\R^d/0$ which is homogeneous of degree $1$. We denote by $m(a(D))$ the Fourier multiplier transformation associated with the Fourier multiplier $m(a(\xi))$, i.e.
\[ \ft[ m(a(D))f] (\xi) = m(a(\xi)) \ft f(\xi). \]
\begin{cor}\label{cor:main2}
Let $1 < p < \frac{2(d+1)}{d+3}$, $p\leq q\leq \infty$, and $\psi$ be a non-trivial smooth function compactly supported in $(0,\infty)$. Assume that $\Sigma=\{ \xi\in \R^d: a(\xi) = 1 \}$ has everywhere non-vanishing Gaussian curvature. Then
\begin{equation*}
\norm{m(a(D)) f}_{L^{p,q}(\R^d)} \les \sup_{t>0} \norm{m(t\cdot)\psi}_{B^2_{\alpha(p),q}(\R) }\norm{ f}_{L^p(\R^d)}.
\end{equation*} 
\end{cor}
The result, in the radial case $a(\xi)=|\xi|$, was obtained in \cite{LRS}. A direct proof of Corollary \ref{cor:main2} will be given in the author's thesis. The sharpness of Corollary \ref{cor:main2} (see \cite{LRS}) yields the sharpness of Theorem \ref{thm:main2} in the sense that the Besov space $B^2_{\alpha(p),q}$ cannot be replaced by any larger $L^2$-based Besov space. 

The proof of Theorem \ref{thm:main2} will be based on the atomic decomposition using Peetre's square function and the following uniform estimate.
\begin{thm}\label{thm:main}
Let $1 < p < \frac{2(d+1)}{d+3}$ and $p\leq q\leq \infty$. Assume that $m$ is a function in ${B^2_{\alpha(p),q}(\R) }$ and supported in $[1/2,2]$. Then
\begin{equation*}
\sup_{t>0} \norm{m(A/t) f}_{L^{p,q}(M)} \les \norm{m}_{B^2_{\alpha(p),q}(\R) }\norm{ f}_{L^p(M)}.
\end{equation*}
\end{thm}

Let us briefly discuss some of standard ingredients of the proof. The starting point of our analysis is the representation of $m(A/t)$ using the Fourier transform;
\begin{equation*}
m(A/t)f = \frac{1}{2\pi}\int t\widehat{m}(tr) e^{i r A}f dr,
\end{equation*}
where $e^{irA} f = \sum_l e^{ir\lambda_l} E_l f$ solves the Cauchy problem for $i\partial_r +A$. We distinguish the integral into two parts; $|r|\geq \epsilon$ and $|r|<\epsilon$, for a sufficiently small $\epsilon>0$. To handle the case $|r|\geq \epsilon$, one can replace $L^p$ norm with $L^2$ norm by H\"{o}lder's inequality without any loss by the compactness of $M$. Then we apply an orthogonality argument and Theorem \ref{thm:spec} as in \cite{SS}. We shall see that this part, in fact, behaves better; it is sufficient to assume that $m\in B^2_{\alpha(p),\infty}$. In the case $|r|\leq \epsilon$, there is a parametrix constructed by Lax and H\"{o}rmander (see \cite{Hor}), which provides an approximation of $e^{i r A}$ by Fourier integral operators. This makes it possible to apply arguments for quasiradial Fourier multipliers (cf. \cite{Kim}).

The novelty of this paper perhaps lies in certain quasi-orthogonality estimates, which control the interaction between operators associated with different dyadic pieces of $\widehat{m}$ (see Section \ref{sec:decom}). The proof of the estimates is based on \cite{LRS}, but it requires finer estimates. We employ the so-called second dyadic decomposition (see e.g. \cite{SSS,Stein}) and adapt an idea from the work by Lee and Seeger \cite{LS} for the construction of an exceptional set.

This paper is organized as follows. In Section \ref{sec:pre}, we provide some preliminary standard estimates. In Section \ref{sec:reduction}, we reduce Theorem \ref{thm:main} to a normalized local estimate and further to a restricted weak-type inequality. We prove the restricted weak-type inequality in Section \ref{sec:L1} and \ref{sec:res_weak}. In Section \ref{sec:main2}, we prove Theorem \ref{thm:main2}.

We close this section with a final note on notations. We denote by $N$ a sufficiently large number (with respect to $d$) which may differ from line to line. We use the notations $\Box = O(B)$ and $\Box \les B$ to indicate $|\Box| \leq C B$ for a harmless constant $C$, which is allowed to depend on $M,A,d,p,q,\epsilon,N$.

\subsection*{Acknowledgment}
This paper will be a part of the author's PhD thesis. He would like to thank his advisor Andreas Seeger for his support, guidance and constant encouragement throughout this project. This work was supported in part by the National Science Foundation.

\section{Preliminary estimates} \label{sec:pre}
\subsection{Fourier integral estimates} \label{sec:FIO}
In what follows, we denote by $S(x,y)$ the integral kernel of an operator $S$, and vice versa.
By a compactness argument, we may assume throughout the paper that $f$ is supported in a compact subset $\Omega_0$ of a coordinate patch $\Omega \subset M$. We shall identify $\Omega$ to a relatively compact open subset of $\R^d$. In addition, we shall fix a compact subset $X$ of $\Omega$ whose interior contains $\Omega_0$. 

We recall that, there is $\epsilon>0$ such that if $|r|\leq 100\epsilon$, then $e^{irA}$ admits an approximation by a Fourier integral operator (see e.g. \cite{Hor, Shubin, Sogge}); 
\[ e^{irA}(x,y) = S_r(x,y)  + E_r(x,y), \]
where $E_r(x,y)=E(r,x,y)$ is an ``error" term satisfying 
$|\partial_r^N E(r,x,y)| \leq C_N.$ 

The support of $S_r(x,y)$ can be chosen to be a sufficiently small neighborhood of the diagonal $\{ (x,x): x\in M \}$ by taking sufficiently small $\epsilon$. In particular, we may assume that $S_r f$ is supported in $X$. In addition, we may assume that $S_r \tilde{f}$ is supported in a compact subset of $\Omega$ if $\tilde{f}$ is supported in $X$. Moreover, in local coordinates, $S_r(x,y)$ can be expressed as 
\begin{equation*}
S_r(x,y) = \int e^{i\varphi(x,y,\xi)} e^{i ra(y,\xi)} q(r,x,y,\xi) d\xi,
\end{equation*}
where $q$ is a symbol of order zero in $\xi$ variable. We may assume without loss of generality that $q$ vanishes if $|r|\geq 200\epsilon$. The phase function $\varphi$ satisfies $\varphi(x,y,\xi)=\inn{x-y,\xi} + O(|x-y|^2|\xi|)$ and is homogeneous of degree 1 in the $\xi$ variable. We also note that the principal symbol $a(y,\xi)$ is homogeneous of degree 1 in the $\xi$-variable.

We shall use the following useful kernel estimate due to Seeger to control ``error terms".
\begin{lem}[ {\cite[Proposition 3.1]{Seeger} }] \label{lem:kernel}
Let $b$ be a tempered function such that $\supp\widehat{b} \subset [-2^{j+5},2^{j+5}]$. Assume that $2^j \leq \epsilon t$. 

We have
\begin{equation*}
\abs{\int \widehat{b}(r) E(r/t,x,y) dr } \les \sup_{\lambda \in \R} \frac{|b(\lambda)|}{(1+2^j|\lambda|)^N}.
\end{equation*}
In particular, $b(A/t)(x,y)$ follows the same estimate if $x\notin \Omega$ and $y\in X$.

If $x,y \in \Omega$, then there is a constant $C=C_{a,\epsilon}$ such that the following holds; 
\begin{equation*}
|b(A/t)(x,y)| \les \sup_{\lambda \in \R} \frac{|b(\lambda)|}{(1+2^j|\lambda|)^N}  \left(\frac{1}{|x-y|^{d-1}} + \frac{2^j/t}{|x-y|^{d+1}} \right)
\end{equation*}
provided that $|x-y|\geq C2^j/t$.
\end{lem}

We also need the following result.
\begin{lem}\label{lem:symbol}
Assume that $m$ is a symbol of order $-\delta$ for some $\delta>0$, i.e.
\[ |m^{(n)}(\lambda)| \leq C_n (1+|\lambda|)^{-n-\delta}, \] 
for all $n\geq 0$. Then for $1 \leq p \leq \infty$, we have
\[ \sup_{t>0 }\norm{m(A/t) f}_{L^p(M)} \les \norm{f}_{L^p(M)}. \]
\end{lem}
\begin{proof}
This seems to be a standard result, but we sketch the proof for completeness.
Write $m=\sum_{k\geq0} m^k$, where $m^k$ is $m$ smoothly cut off to the set $\{ \lambda: |\lambda|\sim 2^k \}$ for $k\geq 1$ and $[-2,2]$ for $k=0$. By the triangle inequality, it suffices to prove that 
\begin{equation}\label{eqn:mkt}
\sup_{t>0} \norm{m^k_t(A) f}_{L^p(M)} \les 2^{-k\delta} \norm{f}_{L^p(M)},
\end{equation}
where $m^k_t = m^k(\cdot/t)$. We may assume that $f$ is supported in a compact subset of a coordinate patch $\Omega\subset M$.

Note that if $0<t \leq 2^{-k}$, then $m^k_t(A)$ involves only eigenvalues bounded by $O(1)$. Thus, $m^k_t(A)(x,y)$ is bounded by $O(\norm{m^k}_\infty)$ which is $O(2^{-k\delta})$. This can be verified by Weyl's formula and sup-norm bounds for eigenfunctions (see e.g. \cite{Sogge}). Therefore, we may assume that the sup is taken over $t > 2^{-k}$ in \eqref{eqn:mkt}. 

Observe that for $k,n\geq  0$, $m^k_t$ satisfies
\[ (2^kt)^n \abs{ \left(\frac{d}{d\lambda}\right)^n {m^k_t} (\lambda)} \les 2^{-k\delta}. \]
An examination of the proof of \cite[Lemma 2.4]{SS} (see also \cite{Taylor}) shows that we may write
\[ m^k_t(A) = S^k_t + R^k_t, \]
where $R^k_t$ is a negligible error term, and the kernel of $S^k_t$ satisfies (in local coordinates)
\[ |S^k_t(x,y)|  \les 2^{-k\delta} (2^kt)^d (1+2^kt|x-y|)^{-N}. \]
This proves that the $L^1(\Omega)$ and $L^\infty(\Omega)$ operator norms of $S^k_t$ are $O(2^{-k\delta})$, which yields \eqref{eqn:mkt} by interpolation.
\end{proof}

\subsection{$L^p \to L^2$ estimates}
Let $\eta_j= \eta(\cdot/2^j)$ for a fixed even smooth bump function $\eta$ supported in $\{\lambda : |\lambda| \in [1/4,4] \}$ for $j\geq 1$. We shall further assume that $\eta$ is $1$ on $\{\lambda : |\lambda| \in [1/2,2] \}$, but this is not required for the following estimate. 
\begin{lem} \label{lem:basiclp}
Let $\beta$ be an $L^2$ function supported on $\{\lambda : |\lambda| \in [1/8,8] \}$ and $j\geq 1$. Then
\begin{equation}
\norm{\beta*\check{\eta_j} (A/t) f}_{L^2(M)} \les t^{\delta(p)} \max(t^{1/2}, 2^{j/2}) \norm{\beta}_{L^2}\norm{f}_{L^p(M)}.
\end{equation}
\end{lem}
\begin{proof}
When $\beta*\check{\eta_j}$ is multiplied by a compactly supported function, then the estimate is given in \cite{SS}. To handle our case, we add one more standard error estimate. 

We need the following result obtained by the Plancherel-Polya lemma (see \cite[Equation (3.12)]{SS}).
\begin{equation}\label{eqn:PPin}
\Big( \sum_{0 \leq n \leq  16t} \sup_{\lambda_l \in [n,n+1]} |\beta*\check{\eta_j} (\lambda_l/t)| ^2 \Big)^{1/2} \les \max( t^{1/2},2^{j/2}) \norm{\beta*\check{\eta_j}}_{L^2}.
\end{equation}
By orthogonality, we may bound $\norm{\beta*\check{\eta_j} (A/t) f}_{L^2(M)}$ by
\begin{align*}
&\leq \Big( \sum_n \sum_{\lambda_l\in [n,n+1]} |\beta*\check{\eta_j}(\lambda_l/t)|^2 \norm{E_l f}_{L^2(M)}^2 \Big)^{1/2} \\
&\leq \Big( \sum_n \sup_{\lambda_l\in [n,n+1]} |\beta*\check{\eta_j}(\lambda_l/t)|^2 \norm{\chi_n f}_{L^2(M)}^2 \Big)^{1/2} \\
&\les \Big( \sum_n (1+n)^{2\alpha(p)-1} \sup_{\lambda_l\in [n,n+1]} |\beta*\check{\eta_j}(\lambda_l/t)|^2 \Big)^{1/2} \norm{f}_{L^p(M)}.
\end{align*}
Note that $\beta*\check{\eta_j}$ is essentially supported in $[1/16,16]$; if $|\lambda| \notin [1/16,16]$, then 
\[
|\beta*\check{\eta_j}(\lambda)| \les 2^{-jN}\norm{\beta}_1 (1+2^j|\lambda|)^{-N}.
\]
Therefore, if $\lambda_l \sim 2^k t$ for sufficiently large $k$, say $k\geq 4$, then
\[ 
|\beta*\check{\eta_j}(\lambda_l/t)| \les 2^{-jN} 
2^{-kN} \norm{\beta}_1.
\]
Thus, we may bound $\norm{\beta*\check{\eta_j} (A/t) f}_{L^2(M)}$ by a constant times
\begin{align*}
&t^{\delta(p)} \Big( \sum_{n\leq 16t}  \sup_{\lambda_l\in [n,n+1]} |\beta*\check{\eta_j}(\lambda_l/t)|^2 \Big)^{1/2} \norm{f}_{L^p(M)}\\
&+ \Big( \sum_{k\geq 4} (2^kt)^{2\alpha(p)-1} \sum_{2^k t \leq n \leq 2^{k+1}t}  \sup_{\lambda_l\in [n,n+1]} |\beta*\check{\eta_j}(\lambda_l/t)|^2 \Big)^{1/2} \norm{f}_{L^p(M)} \\
&\les t^{\delta(p)}\max( t^{1/2},2^{j/2}) \norm{\beta}_{L^2} \norm{f}_{L^p(M)}.
\end{align*}
Here, we have used \eqref{eqn:PPin} and $\norm{\beta*\check{\eta_j}}_{2} \les \norm{\beta}_{2}$.
\end{proof}

\subsection{Dyadic Decomposition} \label{sec:bdcom}
We first remark that if $m\in L^\infty([1/2,2])$, then
\[
\sup_{0<t\leq 1}\norm{m(A/t)f}_{L^\infty(M)}  \les \norm{m}_{\infty} \norm{f}_{L^1(M)},
\]
which automatically implies $L^p$ estimates by H\"{o}lder's inequality (see the proof of Lemma \ref{lem:symbol}). Thus, we may restrict our attention to $t\geq 1$.

Let $\phi$ be a smooth non-negative even function supported on $\{r : |r|\leq 2 \}$ which is $1$ on $\{r: |r|\leq 1 \}$. Set $\phi_j (r) = \phi(r/2^j) - \phi(r/2^{j-1})$. It follows that $\phi_j$ is supported on $I_j := \{r: |r| \in [2^{j-1},2^{j+1}] \}$ and that $\sum_{j\geq 1} \phi_j(r) =1$ if $|r|\geq 2$. Let $\phi_0(r) = 1- \sum_{j\geq 1} \phi_j(r)$.

We decompose $m$ as $\sum_{j\geq 0} m_j$, where $\widehat{m_j} = \wh{m} \phi_j$. Note that by Lemma \ref{lem:basiclp}, we have
\begin{equation} \label{eqn:mjb}
\norm{ m_j(A/t) f }_{L^2(M)} \les t^{\delta(p)} \max( t^{1/2},2^{j/2}) \norm{ m_j}_2 \norm{f }_{L^p(M)}.
\end{equation}

By compactness and \eqref{eqn:mjb}, we may bound $\bnorm{\sum_{2^j \geq \epsilon t} m_j (A/t) f }_{L^p(M)}$ by a constant times
\begin{equation} \label{eqn:restest}
\begin{split}
\bnorm{\sum_{2^j \geq \epsilon t} m_j (A/t) f }_{L^2(M)} &\les t^{\delta(p)} \sum_{2^j \geq \epsilon t}  2^{j/2} \norm{m_j}_2 \norm{f }_{L^p(M)}\\
&\les t^{\delta(p)} \sum_{2^j \geq \epsilon t} 2^{-j(\alpha(p)-1/2)} \norm{m}_{B^2_{\alpha(p),\infty}} \norm{f }_{L^p(M)} \\ 
&\les \norm{m}_{B^2_{\alpha(p),\infty}} \norm{f }_{L^p(M)}.
\end{split}
\end{equation}

Moreover, we have $\norm{m_0 (A/t)}_{L^p(M) \to L^p(M)} \les \norm{m}_2$ by Lemma \ref{lem:symbol}. Therefore, for the proof of Theorem \ref{thm:main}, it remains to show that
\begin{equation} \label{eq:ulgo1}
\norm{ \sum_{1< 2^j < \epsilon t} m_j(A/t) f}_{L^{p,q}(M)} \les \norm{m}_{B^2_{\alpha(p),q} }\norm{ f}_{L^p(M)},
\end{equation}
with an implicit constant uniform in $t\geq 1$ with the $\epsilon$ fixed in Section \ref{sec:FIO}. In what follows, $\sum_{1< 2^j < \epsilon t}$ shall be often abbreviated to $\sum_j$.

\section{Further reductions} \label{sec:reduction}
In this section, we reduce \eqref{eq:ulgo1} to a certain restricted weak-type inequality.
\subsection{Reduction to a local estimate} \label{sec:local}
We shall use a cut-off function $\tilde{\eta}$ supported in $\{\lambda : |\lambda| \in [1/16,16] \}$ which is $1$ on $\{\lambda : |\lambda| \in [1/8,8] \}$. Then we may write \[ \tilde{\eta}(A/t) = S_t  + R_t,\] where $\norm{R_t f}_p \les t^{-N} \norm{f}_p$ and $S_tf$ is supported in $X$. Moreover, $|S_t(x,y)| \les t^d(1+t|x-y|)^{-N}$ in local coordinates. We refer the reader to \cite[Lemma 2.4]{SS} for details.

The purpose of this subsection is to reduce \eqref{eq:ulgo1} to the following proposition.
\begin{prop} \label{prop:main} Let $1 \leq p < \frac{2(d+1)}{d+3}$ and assume that $b_j$ satisfies the following conditions;
\begin{enumerate}[(i)]
\item $\norm{b_j}_2 \leq C$ for all $j$.
\item ${\supp \widehat{b_j}} \subset \{ r: |r| \in [2^{j-2},2^{j+2}] \}$.
\item For $n,N \geq 0$, we have
\begin{equation}
|b_j^{(n)}(\lambda)| \leq C_{n,N} 2^{-jN} (1+2^j|\lambda|)^{-N}, \;\;\text{ if }\; |\lambda| \notin [1/8,8].
\end{equation}
\end{enumerate}Then we have
\begin{equation} \label{eqn:goal0}
\bnorm{\sum_{1< 2^j < \epsilon t} 2^{jd/2} {b_j}(A/t) S_t f_j }_{L^p(\Omega)} \les \Big(\sum_j 2^{jd} \norm{f_j}_{L^p(\Omega)}^p \Big)^{1/p},
\end{equation}
whenever the support of $f_j$ is contained in a fixed compact subset $\Omega_0$ of a coordinate patch $\Omega\subset M$.
\end{prop}

From Proposition \ref{prop:main}, we may derive the following result.
\begin{prop} \label{thm:ma} Let $p$, $b_j$, and $f_j$ as in Proposition \ref{prop:main}. Then,
\begin{equation} \label{eqn:goalglobal}
\bnorm{\sum_{1< 2^j < \epsilon t} 2^{jd/2} {b_j}(A/t)f_j }_{L^p(M)} \les \Big(\sum_j 2^{jd} \norm{f_j}_{L^p(\Omega)}^p \Big)^{1/p}.
\end{equation}
Moreover, for $1<p< \frac{2(d+1)}{d+3}$ and $p \leq q \leq \infty$, we have
\begin{equation}\label{eqn:lorentzglobal}
\bnorm{\sum_{1< 2^j < \epsilon t} 2^{-jd(\frac{1}{p}-\frac{1}{2})} {b_j}(A/t) f_j}_{L^{p,q}(M)} \les \bnorm{\Big(\sum_j |f_j|^q \Big)^{1/q}}_{L^p(\Omega)}.
\end{equation}
\end{prop}
\begin{proof}

To see \eqref{eqn:goalglobal}, we note that by Lemma \ref{lem:kernel},
\[
\norm{{b_j}(A/t) f}_{L^p(\Omega^c)} \les 2^{-jN} \norm{f}_{L^p(\Omega)},
\]
for $f$ compactly supported in $\Omega$. Since $S_t f$ is compactly supported in $\Omega$, we obtain
\begin{equation} \label{eqn:goalglobal1}
\bnorm{\sum_{1< 2^j < \epsilon t} 2^{jd/2} {b_j}(A/t) S_t f_j }_{L^p(\Omega^c)} \les \Big(\sum_j 2^{jd} \norm{f_j}_{L^p(\Omega)}^p \Big)^{1/p}
\end{equation}
by the triangle inequality and H\"{o}lder's inequality. This enables us to replace $\Omega$ in \eqref{eqn:goal0} by $M$. Next, we may harmlessly replace $S_t$ in \eqref{eqn:goal0} by $\tilde{\eta}(A/t) = S_t  + R_t$ since the $L^p$ operator norm of $R_t$ is $O(t^{-N})$. Thus, so far we have seen that \eqref{eqn:goal0} implies 
\begin{equation} \label{eqn:goalglobal2}
\bnorm{\sum_{1< 2^j < \epsilon t} 2^{jd/2} {b_j}\tilde{\eta}(A/t) f_j }_{L^p(M)} \les \Big(\sum_j 2^{jd} \norm{f_j}_{L^p(\Omega)}^p \Big)^{1/p},
\end{equation}
as $b_j(A/t)\tilde{\eta}(A/t) = b_j \tilde{\eta}(A/t)$. 

Note in addition that 
\begin{equation*}
\norm{{b_j}(1-\tilde{\eta})(A/t) f}_{L^p(M)} \les 2^{-jN} \norm{f}_{L^p(M)},
\end{equation*}
which follows from Lemma \ref{lem:symbol} and the fact that
\begin{equation} \label{eqn:outpo}
 | [b_j (1-\tilde{\eta}) ]^{(n)} (\lambda)| \leq C_{n,N} 2^{-jN} (1+2^j|\lambda|)^{-N}.
\end{equation}

Taking account of this, we have 
\begin{equation} \label{eqn:goalglobal3}
\bnorm{\sum_{1< 2^j < \epsilon t} 2^{jd/2} {b_j}(1-\tilde{\eta})(A/t) f_j }_{L^p(M)} \les \Big(\sum_j 2^{jd} \norm{f_j}_{L^p(\Omega)}^p \Big)^{1/p},
\end{equation}
which, together with \eqref{eqn:goalglobal2}, implies \eqref{eqn:goalglobal}.

The Lorentz space estimates \eqref{eqn:lorentzglobal} follows from \eqref{eqn:goalglobal} by an interpolation lemma \cite[Lemma 2.4]{LRS}.
\end{proof}

Next, we show that Proposition \ref{thm:ma} implies \eqref{eq:ulgo1}. Let $\eta_0$ be a smooth bump function supported in $\{\lambda : |\lambda| \in [1/6,6] \}$ which is $1$ on $\{\lambda : |\lambda| \in [1/4,4] \}$. We decompose $m_j$ as
\begin{equation}\label{eqn:decomp}
\widehat{m_j} = \widehat{m_j} \eta_j = 
\ft[m_j \eta_0] \eta_j + \ft[m_j (1-\eta_0)] \eta_j.
\end{equation}

If we let $h_j = [m_j (1-\eta_0)] * \check{\eta_j}$, then $h_j^{(n)}$ satisfies estimates similar to \eqref{eqn:outpo}, up to a multiplicative constant $\norm{m}_1$. Therefore, we have 
\[ \norm{ h_j(A/t)f }_{L^p(M)} \les 2^{-jN} \norm{m}_{L^1} \norm{f}_{L^p(M)},
\]
which reduces \eqref{eq:ulgo1} to the following estimate:
\begin{equation}\label{eq:ulgo2}
\norm{ \sum_j [(m_j \eta_0) *  \check{\eta_j}] (A/t) f}_{L^{p,q}(
M)} \les \norm{m}_{B^2_{\alpha(p),q} }\norm{ f}_{L^p(\Omega)}.
\end{equation}

Set $b_j = (m_j \eta_0) *  \check{\eta_j}/\norm{m_j}_2$ ($b_j=0$ when $\norm{m_j}_2=0$). Then $b_j$ satisfies the assumptions in Proposition \ref{prop:main} and \eqref{eq:ulgo2} follows by applying \eqref{eqn:lorentzglobal} with $f_j = 2^{j\alpha(p)}\norm{m_j}_2 f$. Thus, it only remains to prove Proposition \ref{prop:main} for the proof of Theorem \ref{thm:main}.

\subsection{Reduction to a restricted weak-type inequality}
Fix a relatively compact open set in $\R^d$ which we identify with $\Omega \subset M$. For each $z=(z_1,\cdots,z_d)\in  \Z_t^d := (\Z/t)^d$, set $q_z = X \cap \prod_{i=1}^d [z_i,z_i + 1/t)$. Let $f_{j,z}$ be a function such that $|f_{j,z}| \leq \chi_{q_z}$. Suppose that we have the following;
\begin{equation}\label{eqn:goal2}
\bnorm{\sum_j 2^{jd/2} \sum_{z\in \Z_t^d} \gamma_{j,z} {b_j}(A/t) f_{j,z} }_{L^p(\Omega)} \les  \Big(\sum_j 2^{jd} t^{-d} \sum_{z\in\Z_t^d} |\gamma_{j,z}|^p \Big)^{1/p},
\end{equation}
where the implicit constant is independent of the choice of $f_{j,z}$ and $\gamma_{j,z}$. 

We claim that Proposition \ref{prop:main} follows from \eqref{eqn:goal2}. To see this, define 
\begin{align*}
\gamma_{j,z} = \sup_{x\in q_z}|S_t f_j(x)|  \;\; \text{ and} \;\;
f_{j,z} = \gamma_{j,z}^{-1} \chi_{q_z} S_t f_j  \; \text{ if $\gamma_{j,z} \neq 0$,}
\end{align*}
and $f_{j,z}=0$ if $\gamma_{j,z}=0$. Since $\sum_{z\in \Z_t^d} \gamma_{j,z}f_{j,z} = S_t f_j$, it remains to observe that 
$\sum_{z\in \Z_t^d} |\gamma_{j,z}|^p \les t^d\norm{f_j}_p^p.$ Let $u_t = t^{-d}(1+t|\cdot|)^{-N}$. Then the observation follows from 
\begin{equation*}
|\gamma_{j,z}|^p \les |u_t*f_j(z)|^p \les t^d \int_{q_z}   |u_t*f_j(x)|^p dx,
\end{equation*}
since $|S_t(x,y)| \les u_t(z-y)$ if $x\in q_z$. 

We shall further reduce \eqref{eqn:goal2} to a restricted weak-type inequality. Let $\mu_d$ be the measure on $\N\times \Z_t^d$ given by 
\begin{equation*}
\mu_d (E) = \sum_{j \geq 1} 2^{jd}t^{-d} \#\{z:(j,z)\in E\},
\end{equation*}
and $T$ be the operator acting on functions on $\N \times \Z_t^d$ by
\begin{equation*}
T\gamma(x) = \sum_j 2^{jd/2} \sum_{z\in \Z_t^d} \gamma_{j,z} {b_j} (A/t)  f_{j,z},
\end{equation*}
for a fixed collection of $f_{j,z}$ satisfying $|f_{j,z}| \leq \chi_{q_z}$. Then \eqref{eqn:goal2} is equivalent to
\begin{equation}\label{eqn:goal3}
\norm{T\gamma}_{L^p(\Omega)} \les \norm{\gamma}_{L^p(\mu_d)}.
\end{equation}

We finish this subsection with an outline of the remaining proof. We first show \eqref{eqn:goal3} for $p=1$ in Section \ref{sec:L1}. After that, we prove the following restricted weak-type inequality in Section \ref{sec:res_weak}; for $1<p<\frac{2(d+1)}{d+3}$ and $\lambda>0$,
\begin{equation} \label{eqn:weak}
\meas \{x\in \Omega: \babs{\sum_j 2^{jd/2} \sum_{z\in \E_j} {{b_j}}(A/t) f_{j,z}}>\lambda \} \les \lambda^{-p} \sum_{j} 2^{jd}t^{-d} \# \E_j,
\end{equation}
where $\E_j$ is a finite subset of $\Z_t^d$. Then, \eqref{eqn:goal3} is verified for $1\leq p < \frac{2(d+1)}{d+3}$ by real interpolation. Here, we may assume that $\lambda>C$ for some large constant $C$ by the $L^1$ estimate.

\section{$L^1$ and pointwise estimates}\label{sec:L1}
In this section, we begin the proof of \eqref{eqn:goal3} for $p=1$. This involves certain pointwise estimates which will be used again for the proof of \eqref{eqn:weak}.
\subsection{Error estimates} \label{sec:error}
For the proof of \eqref{eqn:goal3} for $p=1$, by the triangle inequality, it suffices to show that for any $y\in X$,
\begin{equation}\label{eqn:goalL10}
\int_{\Omega} |{b_j}(A/t)(x,y)| dx \les 2^{jd/2}.
\end{equation}

Note that $\norm{{b_j}}_{L^\infty} \leq \norm{\widehat{b_j}}_{L^1} \les 2^{j/2}$ by the Cauchy-Schwarz inequality and the $L^2$ normalization of $b_j$. Let $\zeta_j(x,y)$ be the characteristic function of $\{ (x,y)\in \Omega \times \Omega: |x-y| \leq C 2^j/t \}$, where $C$ is a constant as in Lemma \ref{lem:kernel}. Then Lemma \ref{lem:kernel} gives the following pointwise estimate:
\begin{equation}\label{eqn:kernel}
|{b_j}(A/t)(x,y)|(1-\zeta_j(x,y)) \les 2^{-jN} \left(\frac{1}{|x-y|^{d-1}} + \frac{2^j/t}{|x-y|^{d+1}} \right).
\end{equation}
Therefore, for any $y\in X$, we have
\begin{equation*}
\int_{\Omega} |{b_j}(A/t)(x,y)| (1-\zeta_j(x,y)) dx \les 2^{-jN}.
\end{equation*}

It suffices to consider the case $|x-y|\leq C 2^j/t$. Recall that
\begin{equation}\label{eqn:ker}
{b_j}(A/t)(x,y) = \iint \widehat{b_j}(r)q(r/t,x,y,\xi)e^{ira(y,\xi)/t} dr e^{i\varphi(x,y,\xi)} d\xi + O(2^{-jN}),
\end{equation}
by Lemma \ref{lem:kernel}. The error term $O(2^{-jN})$ is negligible for the $L^1$ estimate.

We would like to isolate a further error term as in \cite{CS}. Let $\tilde{\eta}$ be a smooth function supported on $\{\xi:(2C)^{-1} \leq |\xi| \leq 2C \}$ which is $1$ on $\{ \xi:C^{-1} \leq |\xi| \leq C \}$ for a large fixed constant $C=16 C_0$; we take a constant $C_0$ such that $C_0^{-1} |\xi| \leq |a(y,\xi)| \leq C_0|\xi|$ for all $y\in M$. We insert $\zeta_j(x,y)\tilde{\eta}(\xi/t)$ and $\zeta_j(x,y)(1-\tilde{\eta}(\xi/t))$ in the double integral in \eqref{eqn:ker} and then call the new integrals as $I$ and $II$, respectively. 

We claim that $II=O(2^{-jN} t^d)$, which is acceptable for the $L^1$ estimate since \[ \meas\{x:|x-y|\les 2^j/t\} = O(2^{jd}/t^d).\] Let $\check{q}$ denotes the Fourier inverse transform of the symbol $q$ in $r$-variable. Observe that 
\begin{equation*}
|[b_j * t \check{q}(t\cdot)] (\lambda)| \les \int \frac{|b_j(\tau)|t d\tau}{(1+t|\lambda-\tau|)^{2N}} \les 2^{-jN}(1+2^j|\lambda|)^{-N}
\end{equation*}
if $|\lambda| \notin [1/16,16]$. This follows by considering the cases $|\tau| \in [1/8,8]$ and $|\tau| \notin [1/8,8]$ separately. Therefore, we may bound $|II|$ by a constant times
\begin{equation*}
\int_{|\xi| \notin [C^{-1} t,Ct]} |[b_j * t \check{q}(t\cdot)] (a(y,\xi)/t) | d\xi \les \int \frac{2^{-jN}}{(1+2^j|a(y,\xi)/t|)^N} d\xi\les 2^{-jN}t^d
\end{equation*}
as claimed.

Define
\[Q_j(r,t,x,y) = \zeta_j(x,y) \int e^{it \Phi(r,x,y,\xi) } q(r,x,y,t\xi)\tilde{\eta}(\xi) d\xi, \]
where \[ \Phi(r,x,y,\xi)  = \varphi(x,y,\xi) + r a(y,\xi). \] 
Then the main term $I:=I_j$ can be written as (after the change of variables $\xi \to t\xi$ and $r\to tr$) 
\begin{equation} \label{eqn:reductionI}
I_j(x,y) = t^d \int Q_j(r,t,x,y) \widehat{b_j}(tr) t dr.
\end{equation}

Let us summarize the reductions we have done so far. 
\begin{lem} \label{lem:reduction}
Let $x,y \in \Omega$. Then 
\begin{align*}
b_j(A/t)(x,y) = I_j(x,y) + E_j(x,y),
\end{align*}
where the $L^1(\Omega)$ operator norm of $E_j$ is $O(2^{-jN})$.
\end{lem}

We remark that, more generally, $L^p(\Omega)$ operator norm of $E_j$ is $O(2^{-jN})$ for $1 \leq p\leq \infty$. Our next goal is to show that for $|r| \sim 2^j/t$ and $y\in X$,
\begin{equation} \label{eqn:L1}
\norm{Q_j(r,t,\cdot,y)}_{L^1(\Omega)} \les 2^{j(d-1)/2}t^{-d},
\end{equation}
which would immediately imply that 
\[ \int_{\Omega} |I_j(x,y)| dx \les 2^{j(d-1)/2} \norm{\widehat{b_j}}_{1}\  \les 2^{jd/2}, \]
giving \eqref{eqn:goalL10}.

\subsection{Proof of \eqref{eqn:L1}} \label{sec:sec}
The proof follows from a standard technique called `second dyadic decomposition' (see \cite{CS,SSS,Stein,Tao}). We give details, closely following a variant used by Tao \cite{Tao}.

First, we may assume, by using a partition of unity and rotation, that $\tilde{\eta}$ is supported in a narrow cone $\{\xi=(\xi',\xi_d): |\xi'| \leq c\xi_d, \;\xi_d \sim 1\}$. Let $\lambda=\xi_d$ and $\lambda\omega =\xi '$, so that $\xi = \lambda(\omega,1)$.

Partition $\omega$ variable into $O(2^{j(d-1)/2})$ many disks $D$ of radius $2^{-j/2}$ centered at $\omega_D$. Let $q_D(r,t,x,y,\xi)$ be the symbol $q(r,x,y,t\xi) \tilde{\eta}(\xi)$ smoothly cut off to the tabular region $\{ \xi: \omega \in D, \lambda \sim 1 \}$ and let
\begin{align*}\label{eqn:Qd}
Q_{j,D}(r,t,x,y) = \zeta_j(x,y) \int e^{it \Phi(r,x,y,\xi) } q_D(r,t,x,y,\xi) d\xi,
\end{align*}
so that $Q_j(r,t,x,y) = \sum_D Q_{j,D}(r,t,x,y)$.

To prove \eqref{eqn:L1}, it suffices to show that for $|r|\sim 2^j/t$,
\begin{equation} \label{eqn:L12}
\norm{Q_{j,D}(r,t,\cdot,y)}_{L^1(\Omega)}\les t^{-d}.
\end{equation}

With an abuse of notation, we shall write 
$\Phi(r,x,y,\omega)$ for $\Phi(r,x,y,(\omega,1))$. To prove \ref{eqn:L12}, we need the following lemma.
\begin{lem} \label{lem:ptbound} $Q_{j,D}(r,t,x,y)$ satisfies the pointwise estimate 
\begin{equation}\label{eqn:ptbound}
\begin{split} &|Q_{j,D}(r,t,x,y)| \les  \\
&2^{-j(d-1)/2}(1+t|\Phi(r,x,y,\omega_D)|)^{-N} (1+2^{-j/2}t|\nabla_{\omega} \Phi(r,x,y,\omega_D) |)^{-N}.
\end{split}
\end{equation}
\end{lem}
Given Lemma \ref{lem:ptbound}, \eqref{eqn:L12} follows from the fact that $|Q_{j,D}(r,t,x,y)|$ rapidly decays, as a function of $x$, away from the ``plate"
\begin{equation*}\label{eqn:set}
P_{j,D}(r,y) = \{ x: |\Phi(r,x,y,\omega_D)| \les  1/t, \;|\nabla_{\omega} \Phi(r,x,y,\omega_D) | \les 2^{j/2}/t \}
\end{equation*}
of measure $O(2^{j(d-1)/2} t^{-d})$. We defer the detail for the moment and proceed the proof of Lemma \ref{lem:ptbound}.
\begin{proof}[Proof of Lemma \ref{lem:ptbound}]
Applying Taylor's expansion, we get
\begin{equation*}
\Phi(r,x,y,\omega) = \Phi(r,x,y,\omega_D) + \nabla \Phi(r,x,y,\omega_D)\cdot (\omega-\omega_D) + E_D(r,x,y,\omega),
\end{equation*}
where (with regarding $\nabla_{\omega}^2 \Phi$ as a quadratic form)
\begin{equation*}
E_D(r,x,y,w) = \frac{1}{2} \nabla_{\omega}^2 \Phi(\omega-\omega_D,\omega-\omega_D) + \text{high order terms}.
\end{equation*}

Next, we note that 
\begin{equation}\label{eqn:symbolbound}
\partial_{\lambda}^n q_D = O(1) \text{  and  } \partial_{\omega}^\alpha q_D = O(2^{j|\alpha|/2}). 
\end{equation}
We claim that the error term $e^{it\lambda E_D}$ can be harmlessly absorbed into the symbol $q_D$; $\tilde{q_D}(r,t,x,y,\lambda,\omega) = e^{it\lambda E_D}q_D$ satisfies bounds similar to those for $q_D$ in \eqref{eqn:symbolbound}. This follows from the fact that estimates similar to \eqref{eqn:symbolbound} hold for $e^{it\lambda E_D}$. To see this, it suffices to show that
\begin{equation}\label{eqn:errorbound}
tE_D(r,x,y,w) = O(1) \text{  and  } t\lambda \partial_{\omega}^\alpha E_D(r,x,y,w) = O(2^{j|\alpha|/2}). 
\end{equation}
Observe that for any multi-index $\alpha$, 
\begin{equation*}
\partial_{\omega}^{\alpha} \Phi = O(|x-y|^2) + O(r) = O(2^j/t),
\end{equation*}
given that $|\xi| \sim 1$ and $|x-y|\les 2^j/t \leq \epsilon$.
This, together with $|\omega-\omega_D| = O (2^{-j/2})$, shows that $tE_D(r,x,y,w) = O(1)$. Next, observe that each $\partial_{\omega}$ derivative removes the gain of $2^{-j/2}$ coming from the factor $\omega-\omega_D$. This gives \eqref{eqn:errorbound}. 
 
We rewrite $Q_{j,D}(r,t,x,y)$ as the product of $\zeta_j(x,y)$ and
\begin{equation}\label{eqn:neat}
\int_{\lambda \sim 1} \int_D e^{it\lambda [\Phi(r,x,y,\omega_D)+ \nabla_{\omega} \Phi(r,x,y,\omega_D)\cdot (\omega-\omega_D) ]} \tilde{q_D}(x,y,\lambda,\omega) d\omega \lambda^{n-1} d\lambda.
\end{equation}
Assume first that $2^{-j/2} |\nabla_{\omega} \Phi(r,x,y,\omega_D)| \geq |\Phi(r,x,y,\omega_D)|/2$. In this case, we integrate by parts in $\omega$-variable to obtain
\begin{equation*}
|Q_{j,D}(r,t,x,y)| \les 2^{-j(d-1)/2} (1+2^{-j/2}t |\nabla_{\omega} \Phi(r,x,y,\omega_D)|)^{-2N},
\end{equation*}
which implies \eqref{eqn:ptbound}. The factor $2^{-j(d-1)/2}$ comes from the measure of the disk $D$.

Next, assume that $|\Phi(r,x,y,\omega_D)|/2 \geq 2^{-j/2} |\nabla_{\omega} \Phi(r,x,y,\omega_D)|$. From this we get $|\Phi(r,x,y,\omega_D)| \geq 2 |\nabla_{\omega} \Phi(r,x,y,\omega_D) \cdot (\omega-\omega_D)|$. Integration by parts in $\lambda$-variable gives 
\begin{equation*}
|Q_{j,D}(r,t,x,y)| \les 2^{-j(d-1)/2} (1+t |\Phi(r,x,y,\omega_D)|)^{-2N},
\end{equation*}
which again implies \eqref{eqn:ptbound}. This completes the proof.
\end{proof}

Finally, in order to obtain \eqref{eqn:L12}, we still need to verify
\begin{equation*}
\int (1+t|\Phi(r,x,y,\omega_D)|)^{-N} (1+2^{-j/2}t|\nabla_{\omega} \Phi(r,x,y,\omega_D) |)^{-N} dx \les 2^{j(d-1)/2} t^{-d}.
\end{equation*}
For a fixed $y$, set $G_i(x) = \partial_{\xi_i} \Phi(r,x,y,\omega_D)$ for $1\leq i \leq (d-1)$, and $G_d(x) = \Phi(r,x,y,\omega_D)$. 
Let $I(z) = (1+t|z_n|)^{-N} (1+2^{-j/2}t|z'|)^{-N}$. Then we have,
\begin{equation*}
2^{j(d-1)/2} t^{-d} \sim \int I(z) dz = \int I(G(x)) |\det G'(x)| dx.
\end{equation*}
Therefore, it remains to show that $|\det G'(x)| \sim 1$.

Recall that by Euler's homogeneity relation, \[G_d(x) = \partial_{\xi_d} \Phi(r,x,y,(\omega_D,1)) + \omega_D \cdot \nabla_{\xi'} \Phi(r,x,y,(\omega_D,1)).\] Form this and \begin{equation*}
\frac{\partial^2}{\partial{x_i}\partial{\xi_j}}  \big( \inn{x-y,\xi} + O(|x-y|^2|\xi|) \big) = \delta_{i,j} + O(\epsilon),
\end{equation*}
we deduce that 
\begin{equation*}
\det G'(x) = \det \left( \frac{\partial^2 \Phi}{\partial{x_i}\partial{\xi_j}} \right) = \det \big( I+ O(\epsilon) \big) \sim 1.
\end{equation*}

\section{Proof of the restricted weak-type inequality \eqref{eqn:weak}} \label{sec:res_weak}
In this section, we finish the proof of Theorem \ref{thm:main} by establishing the restricted weak-type inequality \eqref{eqn:weak}. 
\subsection{Density decomposition}
Let $\mathcal{Q}_j$ be a collection of essentially disjoint cubes of side length $2^j/t$ which cover $\R^d$. Let $\mathcal{Q}_j(\lambda)$ be the collection of all $Q\in \mathcal{Q}_j$ such that $\#\E_j \cap Q > \lambda^p$. Then we decompose $\E_j$ as a union of $\E_j(\lambda)$ and $\bigcup_{Q\in \mathcal{Q}_j(\lambda)} \E_j\cap Q ,$ where $\E_j(\lambda) = \E_j \setminus \bigcup_{Q\in \mathcal{Q}_j(\lambda)} Q$ is the low density part of $\E_j$ satisfying $\E_j(\lambda) \cap Q \leq \lambda^p$ for any $Q\in \mathcal{Q}_j$.

We first deal with the high density part. We claim that
\begin{equation} \label{eqn:weakhigh}
\meas \{x\in \Omega: \babs{ \sum_j 2^{jd/2} \sum_{Q\in \mathcal{Q}_j(\lambda)} {b_j}(A/t) f_{j,Q} } >\lambda \} \les \lambda^{-p} \sum_{j} 2^{jd}t^{-d} \# \E_j,
\end{equation}
where $f_{j,Q} = \sum_{z\in \E_j\cap Q} f_{j,z}$. 

Let $Q^*$ be the cube with the same center as $Q$ but the side length is $5C 2^j/t$, where the constant $C\geq 1$ is chosen as in \eqref{eqn:kernel}. In particular, this ensures that ${b_j}(A/t) f_{j,Q}$ is ``essentially supported" on $Q^*$. Let $E= \bigcup_{j\geq 1} \bigcup_{Q\in \mathcal{Q}_j(\lambda)} Q^*$. Then
\begin{equation*}
|E| \les \sum_j \sum_{Q\in \mathcal{Q}_j(\lambda)} |Q| \leq \sum_j \sum_{Q\in \mathcal{Q}_j(\lambda)} 2^{jd}t^{-d} \lambda^{-p} \# \E_j \cap Q \leq \lambda^{-p} \sum_j 2^{jd}t^{-d} \# \E_j.
\end{equation*}

Therefore, \eqref{eqn:weakhigh} follows from
\begin{equation} \label{eqn:weakhigh2}
\lambda^{-1} \sum_j 2^{jd/2} \bnorm{\sum_{Q\in \mathcal{Q}_j(\lambda)} [{b_j}(A/t) f_{j,Q}]\chi_{(Q^*)^c} }_{L^1(\Omega)} \les \lambda^{-p} \sum_{j} 2^{jd}t^{-d} \# \E_j.
\end{equation}
We remark that the $j$-sum is over $2^{dj} > \lambda^p$ since $\mathcal{Q}_j(\lambda)$ is empty if $2^{dj} \leq \lambda^p$. Therefore, the following lemma implies \eqref{eqn:weakhigh2}.

\begin{lem}\label{lem:error} Let $\mathcal{Q}_j$ be a collection of cubes in $\mathcal{Q}_j$. Then for $1 \leq p\leq \infty$,
\begin{equation*}
\bnorm{ \sum_{Q\in \mathcal{Q}_j} [{b_j}(A/t) f_{j,Q}] \chi_{(Q^*)^c} }_{L^p(\Omega)} \les 2^{-jN} (t^{-d}\# \E_j )^{1/p}.
\end{equation*}
\end{lem}
\begin{proof}
If $x\in \Omega\setminus Q^*$, then 
\begin{align*}
&|{b_j}(A/t) f_{j,Q}(x)|  = \abs{\int {b_j}(A/t)(x,y) f_{j,Q} (y) dy}  \\
&\les 2^{-jN} \int_{ \{ y\in \Omega: |x-y| \geq C 2^j/t \}}  \left(\frac{1}{|x-y|^{d-1}} + \frac{2^j/t}{|x-y|^{d+1}} \right) \sum_{z\in \E_j\cap Q} \chi_{q_z}(y) dy.
\end{align*}
Therefore, we may bound $|\sum_{Q\in \mathcal{Q}_j} {b_j}(A/t) f_{j,Q}(x)\chi_{(Q^*)^c}(x) |$ by
\begin{equation*}
C 2^{-jN} \int_{ \{ y\in \Omega-\Omega: |y| \geq C 2^j/t \}}  \left(\frac{1}{|y|^{d-1}} + \frac{2^j/t}{|y|^{d+1}} \right) \sum_{z\in \E_j} \chi_{q_z}(x-y) dy.
\end{equation*}
The claim for $p=1,\infty$ follows from the above estimate. Finally, interpolation finishes the proof.
\end{proof}

\subsection{$L^2$ Estimates}
By the result of previous subsection, we may assume that 
\begin{equation}\label{eqn:lowd}
\# \E_j \cap Q \leq \lambda^p \;\;\; \text{for any} \;\;\; Q\in \mathcal{Q}_j.
\end{equation}
Let $1\leq p\leq  2(d+1)/(d+3)$. Our goal is to prove the following $L^2$ estimate;
\begin{equation} \label{eqn:L2}
\bnorm{\sum_j 2^{jd/2} \sum_{z\in \E_j} {b_j}(A/t) f_{j,z}}_{L^2(\Omega)}^2 \les \lambda^{2-p} \log_2\lambda \sum_{j} 2^{jd}t^{-d} \# \E_j.
\end{equation}
Given \eqref{eqn:L2}, \eqref{eqn:weak} follows by Chebyshev's inequality.

Let $G_j = \sum_{z\in \E_j} {b_j}(A/t) f_{j,z}$. Then as in \cite{HNS}, we have
\begin{align*}
\begin{split}
\bnorm{\sum_j 2^{jd/2} G_j}_{L^2(\Omega)}^2 
&\les 
\log_2 \lambda \sum_j 2^{jd} \norm{G_j}_{L^2(\Omega)}^2 \\
&+ \sum_{4\log_2 \lambda< k < j-9} 2^{(k+j)d/2} |\inn{G_j,G_k}_{\Omega}|,
\end{split}
\end{align*}
where $\inn{f,g}_{\Omega} = \int_{\Omega} f(x) \overline{g(x)} dx$.

We claim the following estimates;
\begin{align}\label{eqn:L2j}
\norm{G_j}_{L^2(\Omega)}^2 &\les \lambda^{2-p} t^{-d}\# \E_j \\\label{eqn:scalar}
 |\inn{G_j,G_k}_{\Omega}| & \les \lambda^p 2^{jd/2} 2^{-k(d-1/2)} t^{-d} \# \E_j,
\end{align}
which imply \eqref{eqn:L2}. The proof of \eqref{eqn:scalar} will be given in the next subsections.

We proceed to the proof of \eqref{eqn:L2j}. Note that $G_j = \sum_{Q\in \mathcal{Q}_j} {b_j}(A/t) f_{j,Q} \chi_{Q*} + Er_j$, where $\norm{Er_j}_{L^2(\Omega)}^2 \les 2^{-jN} t^{-d} \# \E_j$ by Lemma \ref{lem:error}. Thus, it is enough to estimate $\norm{\sum_{Q\in \mathcal{Q}_j} [{b_j}(A/t) f_{j,Q}] \chi_{Q*}}_{L^2(\Omega)}^2$. By the proof of Lemma \ref{lem:basiclp}, we have 
\begin{equation}\label{eqn:rest}
\norm{{b_j}(A/t) f}_{L^2(M)}^2 \les t^{2\alpha(p)} \norm{f}_{L^p(M)} ^2.
\end{equation}
Then \eqref{eqn:L2j} follows by almost orthogonality, \eqref{eqn:lowd}, and 
\[ \norm{f_{j,Q}}_{p}^2 \les (t^{-d} \# \E_j \cap Q )^{2/p} \leq t^{-2d/p} \lambda^{2-p} \# \E_j \cap Q. \]

\subsection{Quasi-orthogonality estimates} \label{sec:decom}
In this subsection, we prepare for the proof of \eqref{eqn:scalar}. In order to exploit the orthogonality of eigenfunctions, we first replace $\inn {G_j,G_k}_{\Omega} $ by $\inn {G_j,G_k}.$ We claim that
\begin{align*}
|\inn {G_j,G_k}_{\Omega}| &\les |\inn {G_j,G_k}| + |\inn {G_j,G_k}_{M\setminus \Omega}| \\
&\les |\inn {G_j,G_k}| +2^{j/2} 2^{-kN} t^{-d} \# \E_j.
\end{align*}

Indeed, notice that if $x\notin \Omega$, then $e^{irA/t} f(x) = E_{r/t} f(x)$ for $f$ supported on $X$. Let $f_j = \sum_{z\in \E_j} f_{j,z}$. Then we have the following estimates 
\begin{align*}
\norm{\int \widehat{b_k}(r) E_{r/t} f_k dr}_{L^\infty(M\setminus\Omega)} &\les 2^{-kN}\\
\norm{\int \widehat{b_j}(r) E_{r/t} f_j dr}_{L^1(M\setminus\Omega)} &\les \norm{\widehat{b_j}}_{L^1} \norm{f_j}_{L^1(M)} \les 2^{j/2} t^{-d} \# \E_j.
\end{align*}
The first estimate follows from Lemma \ref{lem:kernel} and that $\norm{f_k}_{L^\infty} \les 1$. 

Therefore, it is sufficient to work with $|\inn {G_j,G_k}|$. By orthogonality, we have
\[
\inn{G_j,G_k}= \sum_{z\in \E_j} \inn{ {b_j}(A/t) f_{j,z}, {b_k}(A/t) f_k}= \sum_{z\in \E_j} \inn{ b_{j,k} (A/t) f_{j,z},f_k},
\]
where $b_{j,k} = b_j \overline{b_k}$. Observe that $\widehat{b_{j,k}} =\widehat{b_j} * \overline{\widehat{b_k}(-\cdot)}$ is supported on $\{r:|r|\in [2^{j-3},2^{j+3}] \}$, $\norm{b_{j,k}}_2 = O(2^j)$ and $b_{j,k} (\lambda) = O(2^{-jN}(1+2^j|\lambda|)^{-N})$ if $\lambda \notin [1/16,16]$. Therefore, an examination of the proof of Lemma \ref{lem:reduction} shows that 
\[b_{j,k}(x,y) = I_{j,k}(x,y) + E_{j,k}(x,y), \]
where the $L^1(\Omega)$ operator norm of $E_{j,k}$ is $O(2^{-jN})$ and 
\begin{equation*}
I_{j,k}(x,y)  = t^d \int Q_j(r,t,x,y)  \widehat{b_{j,k}}(tr) tdr.
\end{equation*}

Therefore, by $L^1$ and $L^\infty$ estimates, $ |\inn{G_j,G_k}|$ is bounded by
\begin{align*}
&\sum_{z\in \E_j} |\inn{ I_{j,k} (A/t) f_{j,z},f_{k}}| + O(2^{-jN} t^{-d} \# \E_j) \\
&\leq \sum_{z\in \E_j} \sum_{z' \in \E_k(z)}  |\inn{ I_{j,k} (A/t) f_{j,z},f_{k,z'}}| + O(2^{-jN} t^{-d} \# \E_j),
\end{align*}
where $\E_k(z) = \{ z' \in \E_k : |z'-z| \les 2^j/t \}$. This restriction on the $z'$-sum comes from the fact that $I_{j,k}(x,y)$ is supported in $\{(x,y)\in \Omega\times \Omega:|x-y|\les 2^j/t\}$.

We decompose $b_{j}$ as a sum of $b_{j}^n$ as $n$ ranges over natural numbers comparable to $2^{j-k}$, where $\widehat{b_{j}^n}$ is $\widehat{b_j}$ smoothly cut off to the set $\{ r: |r|\in [(n-1)2^k, (n+1)2^k] \}$. Then $b_{j,k} = \sum_{n\sim 2^{j-k}} b_{j,k}^n$, where $b_{j,k}^n = b_j^n \overline{b_k}$.

So far we have reduced \eqref{eqn:scalar} to the inequality
\begin{align} \label{eqn:scalar2}
\sum_{ n\sim 2^{j-k}} \sum_{z' \in \E_k(z)}  |\inn{ I^n_{j,k} f_{j,z}, f_{k,z'} }| \les \lambda^p 2^{jd/2} 2^{-k(d-1/2)} t^{-d},
\end{align}
where
\begin{equation*}
I_{j,k}^n (x,y)  = t^d \int Q_j(r,t,x,y) \widehat{b_{j,k}^n}(tr) tdr.
\end{equation*}

We borrow from Section \ref{sec:sec} the notations and the estimate for $Q_j(r,t,x,y)$. We recall that \[ Q_j(r,t,x,y)=\sum_D Q_{j,D}(r,t,x,y)\] as $D$ ranges over $O(2^{j(d-1)/2})$ disks covering a compact set in $\R^{d-1}$, i.e. the $\omega$ space. Since $y\in q_z$, $Q_{j,D}(r,t,x,y)$ rapidly decays, as a function of $x$, away from the set $P_{j,D}(r,z)$. Note also that $\widehat{b^n_{j,k}}$ is supported on a set $\{r: |r|=n2^k + O(2^k) \}$. Therefore, for $y\in q_z$, $I^n_{j,k}(x,y)$ is small if $x\notin P^n_{j,k}(z)$, where 
\[
P^n_{j,k}(z) = \bigcup_D \bigcup_{|tr|=n2^k + O(2^k)}  P_{j,D}(r,z).
\]
In other words, it is $z' \in \E_k(z) \cap P^n_{j,k}(z)$, which makes a major contribution to the inner product $|\inn{ I^n_{j,k} f_{j,z}, f_{k,z'} }|$. From now on, without loss of generality, we shall assume that $\widehat{b^n_{j,k}}$ is supported on $\{r: r=n2^k + O(2^k)\}$, handling the other case separately. 

Before we proceed to the proof of \eqref{eqn:scalar2}, we give an informal discussion. We will need to estimate the number of $\E_k(z) \cap P^n_{j,k}(z)$ using the fact that $\E_k(z) \cap Q \leq \lambda^p$ for any $Q \in \mathcal{Q}_k$. In view of the Euclidean case \cite{LRS}, $P^n_{j,k}(z)$ should be a sort of an $O(2^k/t)$ neighborhood of $z+ n2^k/t S^{d-1}$, which may be covered by $O(2^{(j-k)(d-1)/2})$ many ``plates" of dimension $2^{(j+k)/2}/t \times \cdots 2^{(j+k)/2}/t \times 2^k/t$. 

The discussion above motivates the following definitions.  Let $\Theta$ be a maximal $2^{-(j-k)/2}$ separated collection of points in the $\omega$ space (a compact subset of $\R^{d-1}$). Note that $\# \Theta = O(2^{(j-k)(d-1)/2})$. Define an exceptional set $E(z,r) = \bigcup_{\theta\in \Theta} E^\theta_0(z,r)$ (cf. \cite{LS}), where for $l\geq 0$,
\begin{align*}
E^\theta_l(z,r) = \{ x\in \Omega: |\Phi(r,x,z,\theta)| \leq C_{2}  2^{k+l}/t, \;|\nabla_{\omega} \Phi(r,x,z,\theta) | \leq C_{1} 2^{(j+k)/2+l}/t \}
\end{align*}
with constants $C_2 \gg C_1 \gg C_a$ for some $C_a$ to be determined in Section \ref{sec:proofsca2}.

Let $\D_\theta$ be a collection of disks $D$ such that $|\theta - \omega_D| \leq 2^{-(j-k)/2}$ in such a way that for each disk $D$, there is a unique $\theta$ such that $D\in \D_\theta$. For each $\theta\in \Theta$, we have
\[ \Omega \setminus E(z,r)  \subset \Omega \setminus  E^\theta_0(z,r) = \bigcup_{l\geq 1} A^\theta_l(z,r), \]
where $A^\theta_l(z,r)= E^\theta_l(z,r) \setminus E^\theta_{l-1}(z,r)$. 
This gives us the basic splitting
\begin{align*}
\sum_{z' \in \E_k(z)}  |\inn{ I^n_{j,k} f_{j,z}, f_{k,z'} }| &\les \sum_{z' \in \E_k(z)\cap E(z,n2^k/t)}  |\inn{ I^n_{j,k} f_{j,z}, f_{k,z'} }|  \\
&+ \sum_{\theta\in \Theta} \sum_{D\in \D_{\theta}} \sum_{l\geq 1} \sum_{z' \in \E_k(z) \cap A^\theta_l(z,n2^k/t) } |\inn{ I^n_{j,k,D} f_{j,z}, f_{k,z'} }|,
\end{align*}
where $I^n_{j,k,D}$ is the operator with the kernel 
\[  I_{j,k,D}^n (x,y)  = t^d \int Q_{j,D}(r,t,x,y) \widehat{b_{j,k}^n}(tr)t dr. \]

\begin{lem}\label{lem:sc1} Assume that $f_z$ and $f_{z'}$ are functions bounded by the characteristic functions on $q_z$ and $q_{z'}$, respectively. If $z' \in \E_k(z)\cap E(z,n2^k/t)$, then
\begin{equation}\label{eqn:sca1}
|\inn{ I^n_{j,k} f_{z}, f_{z'} }| \les t^{-d}2^{-j(d-1)/2} \norm{\widehat{b^n_j}}_{2}.
\end{equation}

If $z' \in \E_k(z)\cap A^\theta_l(z,n2^k/t)$ and $D\in \D_\theta$, then 
\begin{equation}\label{eqn:sca2}
|\inn{ I^n_{j,k,D} f_{z}, f_{z'} }| \les t^{-d} 2^{-j(d-1)/2} 2^{-(k+l)N} \norm{\widehat{b^n_j}}_1.
\end{equation}
\end{lem}

Given Lemma \ref{lem:sc1}, we can prove \eqref{eqn:scalar2}. Note that $E^\theta_0(z,n2^k/t)$ can be covered by $O(2^{(j-k)(d-1)/2})$ many cubes $Q\in \mathcal{Q}_k$ and $\# \Theta =O(2^{(j-k)(d-1)/2})$. This implies
\begin{align*}
\# \E_k(z) \cap E(z,n2^k/t) \les \lambda^p 2^{(j-k)(d-1)}.
\end{align*}
Therefore, we may estimate the main term by
\begin{align*}
&\sum_{ n\sim 2^{j-k}} \sum_{z' \in \E_k(z)\cap E(z,n2^k/t)}  |\inn{ I^n_{j,k} f_{j,z}, f_{k,z'} }|  \\
&\les \lambda^p t^{-d}2^{j(d-1)/2}  2^{-k(d-1)}\sum_{ n\sim 2^{j-k}} \norm{\widehat{b_j^n}}_2.
\end{align*}

Finally, we get the desired result since
\[ \sum_{ n\sim 2^{j-k}} \norm{\widehat{b_j^n}}_2 \les 2^{(j-k)/2}\big(\sum_{ n\sim 2^{j-k}} \norm{\widehat{b_j^n}}_2^2 \big)^{1/2} \les 2^{(j-k)/2} \norm{\widehat{b_j}}_2\leq 2^{(j-k)/2}. \] 

Next, we note that $E^\theta_l(z,n2^k/t)$ can be covered by $O(2^{(j-k)(d-1)/2}2^{ld})$ many cubes $Q\in \mathcal{Q}_k$ which gives
\begin{align*}
\# \E_k(z) \cap A^\theta_l(z,n2^k/t) \les \lambda^p 2^{(j-k)(d-1)/2}2^{ld}.
\end{align*}
Therefore,
\begin{align*}
&\sum_{ n\sim 2^{j-k}}  \sum_{\theta\in \Theta} \sum_{D\in \D_{\theta}} \sum_{l\geq 1} \sum_{z' \in \E_k(z) \cap A^\theta_l(z,n2^k/t) } |\inn{ I^n_{j,k,D} f_{j,z}, f_{k,z'} }| \\
&\les \lambda^p t^{-d} 2^{-kN} \sum_{D} \sum_{l\geq 1} 2^{l(d-N)} \sum_{ n\sim 2^{j-k}} \norm{\widehat{b^n_j}}_1 \\
&\les \lambda^p t^{-d} 2^{-kN}  2^{j(d-1)/2} \norm{\widehat{b_j}}_1 \les \lambda^p t^{-d} 2^{-kN}  2^{jd/2}
\end{align*}
since there are $O(2^{j(d-1)/2})$ many disks $D$ and $\norm{\widehat{b_j}}_1\les 2^{j/2}\norm{\widehat{b_j}}_2$. This reduces the proof of \eqref{eqn:scalar2} to Lemma \ref{lem:sc1}, which we proceed to prove in the following subsections.

\subsection{Proof of \eqref{eqn:sca1}}
We first observe that if $z' \in \E_k(z)\cap E(z,n2^k/t)$, then $|z'-z|$ is comparable to $2^j/t$. To see this, assume that $|z'-z| \les 2^j/t$ and $z' \in  E_0^\theta(z,n2^k/t)$ for some $\theta \in \Theta$. Then $|\Phi(n2^k/t,z',z,\theta)| \les 2^k/t$, that is, 
\begin{align}\label{eqn:dfe}
|\inn{z'-z,(\theta,1)} + O(|z'-z|^2) + n2^k a(z,(\theta,1))/t|   \les 2^k/t.
\end{align} 
Observe that $|a(y,\xi)| \sim |\xi|$ uniformly in $y\in M$. Since $|z'-z|^2 \les |2^j/t|^2 \les \epsilon 2^j/t$ and $n2^k \sim 2^j$, we conclude that \[ |O(|z'-z|^2) + n2^k a(z,(\theta,1)) /t | \sim 2^j/t. \]
This, combined with \eqref{eqn:dfe}, implies that 
\[ |z'-z| \ges |\inn{z'-z, (\theta,1)}| \ges 2^j/t, \] 
as desired.

Assume that we have
\begin{equation} \label{eqn:sca11}
\begin{split}
&\abs{I(t,x,y) := \int \widehat{\kappa_1 \kappa_2}(r) \int e^{it \phi(x,y,\xi)} e^{ira(y,\xi)} q(r/t,x,y,t\xi) \tilde{\eta}(\xi) d\xi dr} \\
&\les (1+t|x-y|)^{-(d-1)/2}) \norm{\kappa_1}_{L^2}\norm{\kappa_2}_{L^2}.
\end{split}
\end{equation}
With $\kappa_1 = b^n_j$ and $\kappa_2 = \overline{b_k}$ which is $L^2$ normalized, we get
\begin{align*}
|\inn{ I^n_{j,k} f_{z}, f_{z'} }| &\leq \iint t^d | I(t,x,y)| \chi_{q_z}(y) \chi_{q_{z'}}(x) dxdy \\
&\les  (1+t|z-z'|)^{-(d-1)/2} \norm{b^n_{j}}_{2} \iint t^d \chi_{q_z}(y) \chi_{q_{z'}}(x) dxdy \\
&\les 2^{-j(d-1)/2} \norm{b^n_{j}}_{2} t^{-d},
\end{align*}
giving \eqref{eqn:sca1}.

It only remains to prove \eqref{eqn:sca11}. First, by using the generalized polar-coordinate for the cosphere $\Sigma_y = \{ \xi: a(y,\xi) =1 \}$, namely $\xi = \rho \gamma$ for $(\rho,\gamma) \in \R_+\times \Sigma_y$ (see e.g. \cite{Dappa}), we have
\[ 
I(t,x,y) = \iint \widehat{\kappa_1 \kappa_2} (r) e^{ir\rho} H(r/t,x,y,\rho) dr \rho^{d-1} d\rho,
\]
where \[ H(r,x,y,\rho) = \int_{\Sigma_y} q(r,x,y,t\rho \gamma) e^{i\rho \phi(x,y,\gamma)}  \tilde{\eta} (\rho \gamma) d\mu_y (\gamma) \]
for a smooth measure $d\mu_y$ on $\Sigma_y$.

Since $q$ is a symbol of order $0$ and $\phi(x,y,\gamma) = \inn{x-y,\gamma} + O(|x-y|^2|\gamma|)$, one can obtain by our curvature assumption on $\Sigma_y$ and the method of stationary phase that 
\begin{equation*}
|\partial_r^N H(r,x,y,\rho)| \les (1+\rho|x-y|)^{-(d-1)/2},
\end{equation*}
for $N\geq 0$. If we denote by $\check{H}(\eta,x,y,\rho)$ the inverse Fourier transform of $H$ for the $r$-variable, it follows that 
\begin{equation*}
|\check{H}(\eta,x,y,\rho)| \les (1+|\eta|)^{-N} (1+\rho|x-y|)^{-(d-1)/2}.
\end{equation*}

Moreover we note that $|\check{H}(\eta,x,y,\rho)|$ vanishes if $\rho \notin [C^{-1},C]$ for some large constant $C=C_a$ due to the support of $\tilde{\eta}$ and the fact that $|\gamma| \sim 1$ if $\gamma \in \cup_y \Sigma_y$. Thus,
\begin{align*}
&|I(t,x,y)| = \abs{\int_{C^{-1}}^C \int \kappa_1(\rho-\eta) \kappa_2(\rho-\eta) t\check{H}(t\eta,x,y,t\rho) d\eta \rho^{d-1} d\rho} \\
&\les  (1+t|x-y|)^{-(d-1)/2} \int \int_{C^{-1}}^C |\kappa_1 (\rho-\eta) \kappa_2 (\rho-\eta) |d\rho  \frac{t}{ (1+t|\eta|)^{N}} d\eta \\
&\les (1+t|x-y|)^{-(d-1)/2} \norm{\kappa_1}_{2}\norm{\kappa_2}_{2}
\end{align*}
by the Cauchy-Schwarz inequality in $\rho$-variable, giving \eqref{eqn:sca11}.

\subsection{Proof of \eqref{eqn:sca2}} \label{sec:proofsca2}
We may bound $|\inn{ I^n_{j,k,D} f_{z}, f_{z'} }|$ by
\begin{equation} \label{eqn:sca2re}
t^d \iiint |Q_{j,D}(r,t,x,y) f_z(y) f_{z'}(x)| dy  dx |\widehat{b^n_{j,k}}(tr)|tdr.
\end{equation}

We need the following lemma.
\begin{lem}\label{lem:qd}
Let $z' \in A^\theta_l(z,r_0)$ for some $|r_0| \sim 2^j/t$. Then
\begin{equation}\label{eqn:qdestimate}
|Q_{j,D}(r,t,x,y)| \les 2^{-j(d-1)/2}2^{-(k+l)N},
\end{equation}
provided that $|r-r_0|$, $|x-z'|$, and $|y-z|$ are $O(2^k/t)$ and $|\omega_D - \theta| = O(2^{(k-j)/2}).$
\end{lem}
In the integral \eqref{eqn:sca2re}, $r,x,y$ and $\omega_D$ satisfy the conditions in Lemma \ref{lem:qd} with $r_0=n2^k$ by our assumptions. Therefore, we may bound \eqref{eqn:sca2re} by
\begin{equation*}
t^{-d} 2^{-j(d-1)/2} 2^{-(k+l)N} \norm{\widehat{b^n_{j,k}}}_{1}.
\end{equation*}
The proof is finished once we observe that
\[ \norm{\widehat{b^n_{j,k}}}_{1} \les \norm{\widehat{b^n_{j}}}_1 \norm{\widehat{b_k}}_1 \les 2^{k/2}\norm{\widehat{b^n_{j}}}_1. \]

\begin{proof}[Proof of Lemma \ref{lem:qd}]
If $z' \in A^\theta_l(z,r_0)$, then there are two possibilities;
\ben
\item $|\nabla_{\omega} \Phi(n2^k/t,z',z,\theta) | >C_1 2^{(j+k)/2+l-1}/t.$  
\item Condition (i) fails but $|\Phi(n2^k/t,z',z,\theta)| > C_2  2^{k+l-1}/t.$
\een

We first consider the case (i). Note that the magnitudes of the derivatives of $\Phi$ and $\nabla_\omega \Phi$ with respect to $r,x,y$ variables are (much) smaller than $C_a$, where
\[ C_a = C[1+\sup_{y\in M, |\xi|=1, |\alpha|\leq 2} |\partial_{\xi}^{\alpha} a(y,\xi) |] , \]
for a sufficiently large constant $C$. As a result, we have
\begin{equation}\label{eqn:rxy}
|\nabla_{\omega} \Phi(n2^k/t,z',z,\theta) - \nabla_{\omega} \Phi(r,x,y,\theta)| \leq C_a 2^k/t.
\end{equation}

We observe that the magnitudes of the derivatives of $\Phi$ with respect to $\omega$ variables are (much) smaller than $C_a 2^j/t.$ As a result, 
\begin{equation}\label{eqn:omega}
 |\nabla_{\omega} \Phi(r,x,y,\theta) - \nabla_{\omega} \Phi(r,x,y,\omega_D)| \leq C_a 2^{(j+k)/2}/t .
\end{equation}
Combining \eqref{eqn:rxy}, \eqref{eqn:omega}, and (i), we obtain 
\begin{equation*}
 |\nabla_{\omega} \Phi(r,x,y,\omega_D)| \ges 2^{(j+k)/2+l}/t,
\end{equation*}
which implies \eqref{eqn:qdestimate} by Lemma \ref{lem:ptbound}.

The argument is slightly more delicate in the case (ii). First we observe that the failure of (i) implies that
\begin{equation} \label{eqn:omega2}
|\nabla_{\omega} \Phi(n2^k/t,z',z,\omega)| \leq 2 C_1 2^{(j+k)/2+l-1}/t,
\end{equation}
whenever $\omega = \theta + O(2^{(k-j)/2})$.

\eqref{eqn:omega2} implies that
\begin{equation}\label{eqn:omega3}
 |\Phi(n2^k,z',z,\theta) - \Phi(n2^k,z',z,\omega_D)| \leq 4 C_1 2^{k+l-1}/t .
\end{equation}
Moreover, as in \eqref{eqn:rxy}, we have
\begin{equation}\label{eqn:rxy2}
|\Phi(n2^k,z',z,\omega_D) - \Phi(r,x,y,\omega_D)| \leq C_a 2^k/t.
\end{equation}

Finally, our assumptions, \eqref{eqn:omega3}, and \eqref{eqn:rxy2} give
\begin{equation*}
 |\Phi(r,x,y,\omega_D)| \ges 2^{k+l}/t,
\end{equation*}
which implies \eqref{eqn:qdestimate} by Lemma \ref{lem:ptbound}.
\end{proof}

\section{Proof of Theorem \ref{thm:main2}} \label{sec:main2}
\subsection{Reduction to a local estimate}
For proof of Theorem \ref{thm:main2}, we shall use Proposition \ref{prop:main} and the atomic decomposition in \cite{Seeger,LRS}. We first recall the following fact.
\begin{lem}
Let $1\leq q\leq \infty$ and $\alpha>0$. Assume that $\psi_1$ and $\psi_2$ are non-trivial smooth functions with compact supports in $(0,\infty)$. Then
\[
\sup_{t>0} \norm{m(t\cdot)\psi_1}_{B^2_{\alpha,q}} \simeq 
\sup_{t>0} \norm{m(t\cdot)\psi_2}_{B^2_{\alpha,q}},
\]
where the implicit constant is independent of $m$.
\end{lem}
This seems to be well known and can be obtained by modifying the proof of \cite[Lemma 2.4]{GS} used to prove a similar result. By the lemma, we may assume that $\psi$ is a smooth non-negative function supported on $[1/2,2]$, such that 
$\sum_{t\in 2^{\Z}} \psi(r/t)^2= 1$ for $r>0$. We emphasize that, here and in what follows, $t\in 2^{\Z}$ is a \emph{dyadic} number, i.e. $2^n$ for some $n\in \Z$.

Let $m^t = m(t \cdot)\psi$, $m^t_j = m^t * \check{\phi_j}$, and $\sum_t = \sum_{t\in 2^\Z}$. We decompose $m$ as 
\begin{align*}
m &= \sum_{t} m \psi(\cdot/t)^2 = \sum_{t} m^t(\cdot /t) \psi(\cdot/t) = \sum_{t} \sum_j m^t_j (\cdot /t) \psi(\cdot/t) \\
&= \sum_{t} \eta(\cdot/t)\sum_{1\leq 2^j < \epsilon t} m^t_j (\cdot /t) \psi(\cdot/t) + \sum_{t} \eta(\cdot/t) \sum_{2^j \geq \epsilon t} m^t_j (\cdot /t) \psi(\cdot/t)  
\\&= m_I + m_{II}.
\end{align*}
Here we used the fact that $\psi = \eta \psi$. 

Let \[ C_{p,q} = \sup_{t\in 2^\Z} \norm{m^t}_{B^2_{\alpha(p),q}}. \]  
We first handle $m_{II}$. By compactness, (almost) orthogonality and \eqref{eqn:restest}, we may bound $\norm{m_{II}(A)f}_{L^p(M)}$ by a constant times
\begin{align*}
\norm{m_{II}(A) f}_{L^2(M)} 
&\les \bigg( \sum_{t} \bnorm{\sum_{2^j\geq \epsilon t} m^t_j (A/t) \psi(A/t) f}_{L^2(M)}^2  \bigg)^{1/2} \\
&\les \sup_{t>0} \norm{m^t}_{B^2_{\alpha(p),\infty}} 
\bigg(\sum_{t} \norm{\psi(A/t) f}_{L^p(M)}^2  \bigg)^{1/2} \\
&\les C_{p,q} \bnorm{\big(\sum_{t} |\psi(A/t) f|^2\big)^{1/2}}_{L^p(M)} \les C_{p,q}\norm{f}_{L^p(M)},
\end{align*}
where the last inequality follows from the Littlewood-Paley theory (see \cite[Lemma 2.3]{SS}).

Next, we handle the term $m_{I}$. As before, we may assume that $f$ is supported in a compact subset $\Omega_0$ of a coordinate patch $\Omega\subset M$ and fix a compact subset $X$ of $\Omega$, whose interior contains $\Omega_0$.

We claim that
\begin{equation}\label{eqn:lo}
 \bnorm{\sum_{t} \eta(A/t) \sum_{1\leq 2^j < \epsilon t } m^t_j (A/t) \left[ \big(\psi(A/t)f \big) \chi_{X^c} \right] }_{L^p(M)} \les C_{p,q}\norm{f}_{L^p(\Omega)}.
\end{equation}
To see this, note that we have by Lemma \ref{lem:basiclp}
\[ \norm{ m^t_j(A/t) f }_{L^p(M)} \les t^{\alpha(p)} \norm{m^t}_{L^2} \norm{f}_{L^p(M)}. \]
Recall from Section \ref{sec:local} that $\psi(A/t) f(x) = R_{t} f(x)$ for $x \in  X^c$, where the $L^p$ operator norm of $R_{t}$ is $O(t^{-N})$. This gives \eqref{eqn:lo} by the triangle inequality. 

This reduces $\norm{ m_{I}(A)f}_{L^{p,q}(M)} \les C_{p,q}\norm{f}_{L^p(\Omega)}$ to
\begin{equation}\label{eqn:lo2}
 \bnorm{\sum_{t} \eta(A/t) \sum_{1\leq 2^j < \epsilon t } m^t_j (A/t) \big[\big(\psi(A/t)f \big) \chi_{X} \big] }_{L^{p,q}(M)} \les C_{p,q}\norm{f}_{L^p(\Omega)}.
\end{equation}
We proceed to prove \eqref{eqn:lo2} in the following subsections. In what follows, we shall implicitly assume that every $t$-summation is taken over dyadic numbers $t>\epsilon^{-1}$.

\subsection{Atomic decomposition}
We fix a local coordinate for $\Omega$ such that, via the coordinate chart, $\Omega$ contains $\{x \in \R^d: |x|\leq 200d\epsilon \}$ and $X$ is a compact subset of $\{ x\in \R^d: |x|\leq 2^{-20} \epsilon \}$. 

We shall decompose $[\psi(A/t)f] \chi_{X}$ by using the Peetre's square function as in \cite{Seeger,LRS}, which is defined for $x\in X$ by
\[ \ps f(x) = \left( \sum_{t} \sup_{|y|\leq 100d/t} |\psi(A/t) f(x+y)|^2 \right)^{1/2}. \]
We have
\[ \norm{ \ps f}_{L^p(\Omega)} \les \norm{f}_{L^p(\Omega)}  \; \;\text{  for  } 1<p<\infty. \]
See \cite[Lemma 5.1]{Seeger}. In addition, we quote an orthogonality estimate in Lorentz spaces which can be obtained from \cite[Lemma 2.3]{SS} and \cite[Lemma 3.2]{LRS}.
\begin{lem} \label{lem:ortho} Let $1<p<2$ and $p\leq q\leq \infty$. Then 
\[ \norm{\sum_{t} \eta(A/t) f_t}_{L^{p,q}(M)} \les \left(\sum_{t} \norm{f_t}_{L^{p,q}(M)}^p \right)^{1/p}.
\]
\end{lem}

From now on, we closely follow the atomic decomposition from \cite{LRS}. However, in order to avoid repetition, we give a minimal exposition and refer the reader to \cite{LRS} for details. Define 
\begin{align*}
\Lambda_n &= \{ x\in X: \ps f(x) > 2^n \}.
\end{align*}
For each $t\in 2^\Z$, let $\qnk$ be the set of all dyadic cubes $Q$ of sidelength $1/t$ such that 
\[ |Q \cap \Lambda_n | \geq |Q|/2  \;\; \text{ but } \; \; |Q\cap \Lambda_{n+1}| < |Q|/2. \]
Then one may verify that
\[ [\psi(A/t) f] \chi_{X} = \sum_n \sum_{Q\in \qnk} [\psi(A/t) f] \chi_{Q\cap X}. \]

Let
\[ \Lambda_n^* = \{ x\in \R^d : M(\chi_{\Lambda_n})(x) > 100^{-d} \}, \]
where $M$ is the Hardy-Littlewood maximal function. Note that $\Lambda_n^*$ is containted in $\{ x: |x|\leq 2^{-12} \epsilon \}$. There is a Whitney decomposition of $\Lambda_n^*$ as a collection of dyadic cubes $\W_n$. The interior of the cubes in $\W_n$ are disjoint and, for each $Q\in \qnk$, there is a unique $W\in \W_n$ containing $Q$. For each $t\in 2^\Z$, $W \in \W := \cup_n \W_n$, and $n\in \Z$, set \[ a_{t,W,n} = \sumab{Q\in \qnk}{Q\subset W} [\psi(A/t)f] \chi_{Q\cap X}, \]  
and $a_{t,W} = \sum_{n: W\in \W_n} a_{t,W,n}$. For each dyadic number $t$, let $\W(t)$ be the collection of cubes in $\W$ whose sidelength is $t$. Then we have the following decomposition;
\[ [\psi(A/t) f] \chi_{X} = \sum_n \sum_{W\in \W_n} a_{t,W,n} = \sum_{W\in\W} a_{t,W} = \sum_{k\geq 0} \sum_{W \in \W(2^k/t)} a_{t,W}. \] 
Since $\W(2^k/t)$ is empty when $2^{k}/t \geq 2^{-10}\epsilon $, the $k$-sum is taken over $2^{k+10}<\epsilon t$.

To prove \eqref{eqn:lo2}, it is enough to show that
\begin{equation} \label{eqn:lo3}
\bnorm{\sum_{t} \eta(A/t) \sum_{k} \sum_{W \in \W(2^k/t)}  \sum_{1\leq 2^j < \epsilon t } m^t_j (A/t) a_{t,W}}_{L^{p,q}(M)} \les C_{p,q}\norm{\ps f}_{L^p(\Omega)}. 
\end{equation}

Let $m^{t,k} = \sum_{0\leq  j \leq k+10} m^t_j$. Then \eqref{eqn:lo3} is implied by the following claims; there is $\delta>0$ such that for $k\geq 0$,
\begin{align}  \label{eqn:lo5}
\bnorm{\sum_{t} \eta(A/t) \sum_{W \in \W(2^k/t)} \sumab{2^j<\epsilon t}{j>k+10} m^{t}_j (A/t) a_{t,W}}_{L^{p,q}(M)} &\les 2^{-k\delta} C_{p,q}\norm{\ps f}_{L^p(\Omega)}, \\
\label{eqn:lo4}
\bnorm{\sum_{t} \eta(A/t) \sum_{k} \sum_{W \in \W(2^k/t)}  m^{t,k} (A/t) a_{t,W}}_{L^{p}(M)} &\les C_{p,q}\norm{\ps f}_{L^p(\Omega)}.
\end{align}

We continue the proof of \eqref{eqn:lo5} and \eqref{eqn:lo4} in the following subsections. We record here an identity to be used later.
\begin{lem}[{\cite[Equation (57)]{LRS}}] \label{lem:atop}
\begin{equation*} 
\begin{split}
&\sum_{t} \sum_{W \in \W(2^k/t)} \norm{a_{t,W}}_{L^p(M)}^p \\
&\les \sum_{t} \sum_{W \in \W(2^k/t)}  (2^k/t)^{d(\frac{1}{p}-\frac{1}{2})p} \norm{a_{t,W}}_{L^2(M)}^p \les \norm{\ps f}_{L^p(\Omega)}^p.
\end{split}
\end{equation*}
\end{lem}

\subsection{Proof of \eqref{eqn:lo4}} 
We first claim that it is sufficient to prove 
\begin{equation}\label{eqn:lo41}
\bnorm{\sum_{t} \eta(A/t) \sum_{k} \sum_{W \in \W(2^k/t)}  [m^{t,k} (A/t) a_{t,W}] \chi_{\Omega}}_{L^{p}(M)} \les C_{p,q}\norm{\ps f}_{L^p(\Omega)}.
\end{equation} 

Note that $m^{t,k} = m^t *\ift [\phi( \cdot /2^{k+10})]$ and for $x\notin \Omega$ and $y\in X$, we have $|m^{t,k}(A/t)(x,y)| = O(2^{-kN})\norm{m^t}_1$ by Lemma \ref{lem:kernel}. This gives
\begin{align*}
\bnorm{m^{t,k} (A/t) \Big[\sum_{W \in \W(2^k/t)} a_{t,W} \Big] }_{L^{p}(\Omega^c)} &\les 2^{-kN} \norm{m^t}_{1}\bnorm{\sum_{W \in \W(2^k/t)}   a_{t,W}}_{L^p(M)} \\
&\les 2^{-kN} C_{p,q}\Big( \sum_{W \in \W(2^k/t)} \norm{a_{t,W}}_{L^p(M)}^p\Big)^{1/p}, 
\end{align*}
which implies with Lemma \ref{lem:atop} and Lemma \ref{lem:ortho},
\begin{equation}\label{eqn:lerror}
\begin{split}
&\bnorm{\sum_{t} \eta(A/t) \sum_{k} \sum_{W \in \W(2^k/t)}  [m^{t,k} (A/t) a_{t,W}] \chi_{\Omega^c}}_{L^{p}(M)} \\
&\les \sum_{k} \left(\sum_{t} \bnorm{m^{t,k} (A/t) \Big[\sum_{W \in \W(2^k/t)} a_{t,W} \Big] }_{L^{p}(\Omega^c)}^p \right)^{1/p} \\
&\les C_{p,q} \sum_{k} 2^{-kN} \left( \sum_{t} \sum_{W \in \W(2^k/t)} \norm{a_{t,W}}_{L^p(M)}^p \right)^{1/p} \les C_{p,q} \norm{\ps f}_{L^p(\Omega)},
\end{split}
\end{equation}
establishing the claim.

Next, we define operators $T^{t,k}_{sh}$ and $E^{t,k}_{lg}$, by
\begin{align*}
T^{t,k}_{sh}(x,y) &= \chi_{\{ (x,y) \in \Omega\times \Omega: |x-y|\leq C 2^{k+10}/t \} }(x,y) m^{t,k}(A/t)(x,y) \\
E^{t,k}_{lg}(x,y) &= \chi_{\{ (x,y) \in \Omega\times \Omega: |x-y|> C 2^{k+10}/t \} }(x,y) m^{t,k}(A/t)(x,y),
\end{align*}
for a sufficiently large $C>0$ as in Lemma \ref{lem:kernel}.
In order to prove \eqref{eqn:lo41}, it is sufficient to prove the following estimates;
\begin{align} \label{eqn:short}
\bnorm{\sum_{t} \eta(A/t) \sum_{k} \sum_{W \in \W(2^k/t)}  T^{t,k}_{sh} a_{t,W}}_{L^{p}(M)} &\les C_{p,q}\norm{\ps f}_{L^p(\Omega)},\\ \label{eqn:longerror}
\bnorm{\sum_{t} \eta(A/t) \sum_{k} \sum_{W \in \W(2^k/t)}  E^{t,k}_{lg} a_{t,W}}_{L^{p}(M)} &\les C_{p,q}\norm{\ps f}_{L^p(\Omega)}.
\end{align}
By Lemma \ref{lem:kernel}, we see that the $L^p$ operator norm of $E^{t,k}_{lg}$ is $O(2^{-kN} \norm{m^t}_1)$. We omit the proof of \eqref{eqn:longerror}, since it can be shown as in \eqref{eqn:lerror}.

We turn to the proof of \eqref{eqn:short}. We claim that
\begin{equation} \label{eqn:tshl}
\norm{T^{t,k}_{sh} f}_{L^2(M)} \les C_{p,q} \norm{f}_{L^2(M)}.
\end{equation}
This, together with the fact that $T^{t,k}_{sh} a_{t,W}$ is supported in a fixed dilate $W^*$ of $W$, establishes \eqref{eqn:short}. We refer the reader to \cite[Appendix A]{LRS} for details. For the proof of \eqref{eqn:tshl}, it is enough to notice that
\begin{equation*}
\norm{m^{t,k}(A/t) f}_{L^2(M)} \les \norm{m^{t,k}}_{\infty} \norm{f}_{L^2(M)} \les C_{p,q}  \norm{f}_{L^2(M)},
\end{equation*}
since $T^{t,k}_{sh}(x,y) = m^{t,k}(A/t)(x,y)\chi_{\Omega}(x) - E^{t,k}_{lg}(x,y)$ and the $L^p$ operator norm of $E^{t,k}_{lg}$ is $O(2^{-kN}\norm{m^t}_1)$.

\subsection{Proof of \eqref{eqn:lo5}}
We denote $\sumab{2^j<\epsilon t}{j>k+10}$ by $\sum_j$ for the sake of simplicity. We apply Lemma \ref{lem:ortho} and decompose $m^{t}_j$ as in \eqref{eqn:decomp}. Then it is sufficient to show, for the proof of \eqref{eqn:lo5}, that there is $\delta>0$ such that
\begin{equation}\label{eqn:lo6}
\left( \sum_{t} \bnorm{ \sum_j m^{t}_j\eta*\check{\eta_j}(A/t) \Big[ \sum_{W \in \W(2^k/t)} a_{t,W} \Big] }_{L^{p,q}(M)}^p \right)^{1/p} \les 2^{-k\delta}C_{p,q} \norm{\ps f}_{L^p(\Omega)}.
\end{equation}
For the proof of \eqref{eqn:lo6}, we need the following proposition.
\begin{prop}\label{prop:long} Let $1<p< \frac{2(d+1)}{d+3}$, $p\leq q \leq \infty$, and $b_j$ as in Proposition \ref{prop:main}. Fix $k\geq 0$ and assume that $a_z$ is a function supported in $X$, in particular, on the cube $W_z := \prod_{i=1}^d [2^{k} z_i/t,2^k(z_i+1)/t] \cap X$ for each $z\in \Z^d$ and that $\norm{a_z}_{L^2(X)} \leq 1$. Then there is $\delta>0$ such that
\begin{equation} \label{eqn:long}
\begin{split}
\bnorm{ \sumab{2^j<\epsilon t}{j>k+10} 2^{-jd(\frac{1}{p}-\frac{1}{2})} b_j(A/t) \Big[  \sum_z \gamma_{j,z} a_z \Big]}_{L^{p,q}(M)} \\
\les 2^{-k\delta} (2^k/t)^{d(\frac{1}{p}-\frac{1}{2})} \Big( \sum_z \big( \sum_j |\gamma_{j,z}|^q\big)^{p/q}\Big)^{1/p }.
\end{split}
\end{equation}
\end{prop}

We defer the proof of Proposition \ref{prop:long} for the moment and proceed to the proof of \eqref{eqn:lo6}. Fix $t\in 2^\Z$. We can identify $\sum_{W \in \W(2^k/t)} a_{t,W}$ with $\sum_{z}a_{t,W_z}$, where the sum is taken over $z\in \Z^d$ such that $W_z \in \W(2^k/t)$ and $\norm{a_{t,W_z}}_2 \neq 0$. Then set $a_z = \norm{a_{t,W_z}}_2^{-1}a_{t,W_z} $ and $b_j = [m^t_j \eta] * \check{\eta_j}/\norm*{m^t_j}_2$ ($b_j = 0$ if $\norm*{m^t_j}_2=0$). Finally, set $\gamma_{j,z} = 2^{j\alpha(p)} \norm*{m^t_j}_2\norm{a_{t,W_z}}_2$. Then Proposition \ref{prop:long} implies that 
\begin{align*}
&\bnorm{ \sum_j m^{t}_j\eta*\check{\eta_j} \Big[ \sum_{W \in \W(2^k/t)} a_{t,W} \Big] }_{L^{p,q}(M)}^p \\
&\les 2^{-k\delta p} \norm{m^t}_{B^2_{\alpha(p),q}}^p (2^k/t)^{d(\frac{1}{p}-\frac{1}{2})p} \sum_{W \in \W(2^k/t)}\norm{a_{t,W}}_2^p.
\end{align*}
We sum this over $t$ using Lemma \ref{lem:atop}, which gives \eqref{eqn:lo6}.

\begin{proof}[Proof of Proposition \ref{prop:long}]
Define the operator $S_j$ acting on functions on $\Z^d$ by
\[ S_j \gamma_j = 2^{jd/2} b_j(A/t) \big[ \sum_z \gamma_{j,z} a_z \big]. \]
Fix $p_1$ such that $p<p_1< \frac{2(d+1)}{d+3}$. Then we have
\begin{equation}\label{eqn:lp1}
\bnorm{\sum_j S_j \gamma_j}_{L^{p_1}(M)} \les
(2^k/t)^{d(\frac{1}{p_1}-\frac{1}{2})} \left( \sum_j 2^{jd} \norm{\gamma_j}_{l^{p_1}}^{p_1} \right)^{p_1},
\end{equation}
which is a consequence of Proposition \ref{thm:ma}, followed by H\"{o}lder's inequality.

The $2^{-k\delta}$ gain is obtained from an $L^1$ estimate. Indeed, we claim that 
\begin{equation}\label{eqn:l1}
\bnorm{\sum_j S_j \gamma_j}_{L^{1}(M)} \les
2^{k/2}/t^{d/2} \sum_j 2^{jd} \norm{\gamma_j}_{l^{1}}.
\end{equation}
Here, $2^{k/2}$ is a gain over the bound $2^{kd/2}$. Interpolating \eqref{eqn:lp1} and \eqref{eqn:l1} by using \cite[Lemma 2.4]{LRS} gives \eqref{eqn:long}. We refer the reader to \cite{LRS} for details.

By the triangle inequality and Lemma \ref{lem:kernel}, we see that \eqref{eqn:l1} follows from
\begin{equation}\label{eqn:l2}
\norm{b_j(A/t) a_z}_{L^1(\Omega)} \les 2^{k/2} (2^{j}/t)^{d/2}.
\end{equation}
By Lemma \ref{lem:reduction}, we have
\[ \norm{b_j(A/t) a_z}_{L^1(\Omega)} \les \norm{I_j a_z}_{L^1(\Omega)} + 2^{-jN} \norm{a_z}_{L^1(\Omega)}. \]
Since $\norm{a_z}_{L^1(\Omega)} \les (2^k/t)^{d/2}$ and $j>k$, this reduces \eqref{eqn:l2} to 
\begin{equation*}
\norm{I_j a_z}_{L^1(\Omega)} \les 2^{k/2} (2^{j}/t)^{d/2}.
\end{equation*}
This can be further reduced to the statement that we have, for $|r| \sim 2^j/t$ and $j> k+10$,
\begin{equation}\label{eqn:l3}
\norm{t^d \int Q_j(r,t,\cdot,y) a_z(y)dy }_{L^1(\Omega)} \les (2^{j(d-1)}2^{k}/t^{d})^{1/2}.
\end{equation}

Before we proceed to the proof of \eqref{eqn:l3}, we briefly discuss the idea of the proof (cf. \cite{HNS,LRS}). Let $c_z$ be the center of the cube $W_z$. It turns out that the integrand in \eqref{eqn:l3} is essentially supported in the set $E(c_z,r)$, defined in Section \ref{sec:decom}, whose measure is $O(2^{j(d-1)}2^k/t^d)$. We will apply the Cauchy-Schwarz inequality over this set and work with an $L^2$ estimate, which leads us to \eqref{eqn:l3}. 

We first handle the $L^1$ estimate off the set $E(c_z,r)$. As in Section \ref{sec:decom}, 
\begin{align*}
&\norm{t^d \int Q_j(r,t,\cdot,y) a_z(y)dy }_{L^1(\Omega\setminus E(c_z,r))} \\ 
&\leq \sum_D \norm{t^d \int Q_{j,D}(r,t,\cdot,y) a_z(y)dy }_{L^1(\Omega\setminus E(c_z,r))} \\
&\leq \sum_\theta \sum_{D\in \D_\theta} \sum_{l\geq 0} \norm{t^d \int Q_{j,D}(r,t,\cdot,y) a_z(y)dy }_{L^1(A^{\theta}_l(c_z,r))}.
\end{align*}

By Lemma \ref{lem:qd}, we may bound the last line by a constant times
\begin{align*}
&2^{-j(d-1)/2} t^d \sum_\theta \sum_{D\in \D_\theta} \sum_{l\geq 0} 2^{-(k+l)N} |A^{\theta}_l(c_z,r)| \norm{a_z}_{L^1(\Omega)} \\
&\les 2^{-j(d-1)/2} t^d  \sum_\theta \sum_{D\in \D_\theta} \sum_{l\geq 0} 2^{-(k+l)N}  2^{(j+k)(d-1)/2 + k+ld}/t^d  (2^{k}/t)^{d/2}\\
&\les 2^{-kN} (2^{j(d-1)}/t^d)^{1/2}.
\end{align*}

For the main term, we claim that 
\begin{equation} \label{eqn:l2fio}
\norm{t^d \int Q_j(r,t,\cdot,y) f(y)dy }_{L^2(\Omega)} \les \norm{f}_{L^2(\Omega)}.
\end{equation}
Given this, we have
\begin{align*}
&\norm{t^d \int Q_j(r,t,\cdot,y) a_z(y)dy }_{L^1(E(c_z,r))} \\ 
&\les (2^{j(d-1)}2^k/t^d)^{1/2} \norm{t^d \int Q_j(r,t,\cdot,y) a_z(y)dy }_{L^2(\Omega)}
\les (2^{j(d-1)}2^k/t^d)^{1/2}
\end{align*}
by the $L^2$ normalization of $a_z$. 

It only remains to verify \eqref{eqn:l2fio}, i.e. the uniform $L^2(\Omega)$ boundedness of the operator associated with the kernel
\[  \zeta_j(x,y) \int e^{i\Phi(r,x,y,\xi)} q(r,x,y,\xi)\tilde{\eta}(\xi/t) d\xi, \]
for $|r|\sim 2^j/t$. We remark that $q(r,x,y,\xi)\tilde{\eta}(\xi/t)$ is a symbol of order $0$ in $\xi$-variable with bounds uniform in $t$. Moreover, we have
\[ \det \begin{pmatrix} \Phi_{xy} &  \Phi_{x \xi} \\ \Phi_{\xi y} &  \Phi_{\xi \xi}  \end{pmatrix} \neq 0\]
as the phase function $\Phi(r,x,y,\xi)=\varphi(x,y,\xi) + ra(y,\xi)$ is a small perturbation of $\inn{x-y,\xi}$ since $r= O(\epsilon)$ and $|x-y|\les 2^j/t \leq \epsilon$. Therefore, \eqref{eqn:l2fio} follows from the $L^2$ boundedness of Fourier integral operators (see \cite{HorFIO,GrSe}).
\end{proof}

\section{A concluding remark}
It would be interesting to obtain the following maximal version of Theorem \ref{thm:main}; under the same assumptions in Theorem \ref{thm:main},
\begin{equation*}
\norm{\sup_{t> 0} |m(A/t) f|}_{L^{p'}(M)} \les \norm{m}_{B^2_{\alpha(p),q}(\R) }\norm{ f}_{L^{p',q'}(M)}.
\end{equation*}
A related maximal estimate holds for a class of quasiradial Fourier multipliers (see \cite{LRS, Kim}). In view of \cite{LRS}, it would follow from a vector valued version of Theorem \ref{thm:main}, which is likely to be obtained, provided that the following vector valued analogue of Lemma \ref{lem:basiclp} is verified;
\begin{align*}
&\norm{\int_1^2 \beta*\check{\eta_j} (A/st) f_s ds}_{L^2(M)} \\
&\les t^{\delta(p)} \max(t^{1/2}, 2^{j/2}) \norm{\beta}_{L^2}\norm{ \int_1^2 |f_s|ds }_{L^p(M)}.
\end{align*}

\end{document}